\UseRawInputEncoding

\documentclass[12pt, reqno]{amsart}
\usepackage{amssymb}
\usepackage{eucal}
\usepackage{amsmath}
\usepackage{amscd}
\usepackage[dvips]{color}
\usepackage{mathtools}
\usepackage{multicol}
\usepackage[all]{xy}           
\usepackage{graphicx}
\usepackage{color}
\usepackage{colordvi}
\usepackage{xspace}
\usepackage{ulem}
\usepackage{cancel}
\usepackage{tikz}
\usepackage{longtable}
\usepackage{multirow}
\usepackage{ifpdf}
\ifpdf
 \usepackage[colorlinks,final,backref=page,hyperindex]{hyperref}
\else
 \usepackage[colorlinks,final,backref=page,hyperindex,hypertex]{hyperref}
\fi

\begin{document}
\newcommand {\emptycomment}[1]{} 

\newcommand{\tabincell}[2]{\begin{tabular}{@{}#1@{}}#2\end{tabular}}

\newcommand{\nc}{\newcommand}
\newcommand{\delete}[1]{}

\nc{\mlabel}[1]{\label{#1}}  
\nc{\mcite}[1]{\cite{#1}}  
\nc{\mref}[1]{\ref{#1}}  
\nc{\meqref}[1]{Eq.~\eqref{#1}} 
\nc{\mbibitem}[1]{\bibitem{#1}} 

\delete{
\nc{\mlabel}[1]{\label{#1}  
{\hfill \hspace{1cm}{\bf{{\ }\hfill(#1)}}}}
\nc{\mcite}[1]{\cite{#1}{{\bf{{\ }(#1)}}}}  
\nc{\mref}[1]{\ref{#1}{{\bf{{\ }(#1)}}}}  
\nc{\meqref}[1]{Eq.~\eqref{#1}{{\bf{{\ }(#1)}}}} 
\nc{\mbibitem}[1]{\bibitem[\bf #1]{#1}} 
}

\newtheorem{thm}{Theorem}[section]
\newtheorem{lem}[thm]{Lemma}
\newtheorem{cor}[thm]{Corollary}
\newtheorem{pro}[thm]{Proposition}
\newtheorem{conj}[thm]{Conjecture}
\theoremstyle{definition}
\newtheorem{defi}[thm]{Definition}
\newtheorem{ex}[thm]{Example}
\newtheorem{rmk}[thm]{Remark}
\newtheorem{pdef}[thm]{Proposition-Definition}
\newtheorem{condition}[thm]{Condition}

\renewcommand{\labelenumi}{{\rm(\alph{enumi})}}
\renewcommand{\theenumi}{\alph{enumi}}
\renewcommand{\labelenumii}{{\rm(\roman{enumii})}}
\renewcommand{\theenumii}{\roman{enumii}}

\nc{\tred}[1]{\textcolor{red}{#1}}
\nc{\tblue}[1]{\textcolor{blue}{#1}}
\nc{\tgreen}[1]{\textcolor{green}{#1}}
\nc{\tpurple}[1]{\textcolor{purple}{#1}}
\nc{\btred}[1]{\textcolor{red}{\bf #1}}
\nc{\btblue}[1]{\textcolor{blue}{\bf #1}}
\nc{\btgreen}[1]{\textcolor{green}{\bf #1}}
\nc{\btpurple}[1]{\textcolor{purple}{\bf #1}}


\newcommand{\End}{\text{End}}

\nc{\calb}{\mathcal{B}}
\nc{\call}{\mathcal{L}}
\nc{\calo}{\mathcal{O}}
\nc{\frakg}{\mathfrak{g}}
\nc{\frakh}{\mathfrak{h}}
\nc{\ad}{\mathrm{ad}}
\def \gl {\mathfrak{gl}}
\def \g {\mathfrak{g}}

\nc{\ccred}[1]{\tred{\textcircled{#1}}}


\newcommand{\cm}[1]{\textcolor{purple}{\underline{CM:}#1 }}

\newcommand\blfootnote[1]{%
  \begingroup
  \renewcommand\thefootnote{}\footnote{#1}%
  \addtocounter{footnote}{-1}%
  \endgroup
}


\title[Yang-Baxter equations and relative Rota-Baxter operators]
{Yang-Baxter equations and relative Rota-Baxter operators for left-Alia algebras associated to invariant theory}
\author{Chuangchuang Kang}
\address{
School of Mathematical Sciences  \\
Zhejiang Normal University\\
Jinhua 321004 \\
China}
\email{kangcc@zjnu.edu.cn}

\author{Guilai Liu}
\address{Chern Institute of Mathematics \& LPMC    \\
Nankai University \\
 Tianjin 300071   \\
  China}
\email{liugl@mail.nankai.edu.cn}


\author{Shizhuo Yu}
\address{
School of Mathematical Sciences \& LPMC    \\
Nankai University \\
Tianjin 300317              \\
China}
\email{yusz@nankai.edu.cn}

\blfootnote{*Corresponding Author:\ Shizhuo\ Yu. Email:\ yusz@nankai.edu.cn\ }

\begin{abstract}
Left-Alia algebras are a class of algebras with symmetric Jacobi identities. They contain several typical types of algebras as subclasses, and are closely related to the invariant theory.
In this paper, we study the construction theory of left-Alia bialgebras.
We  introduce
 the notion of the left-Alia Yang-Baxter equation.
We show that an antisymmetric solution of the left-Alia Yang-Baxter equation gives rise to a left-Alia bialgebra that we call triangular.
The notions of relative Rota-Baxter operators of left-Alia algebras and pre-left-Alia algebras are introduced to provide antisymmetric solutions of the left-Alia Yang-Baxter equation.
\end{abstract}

\subjclass[2020]{
17A30,  
17A36,
17B38,
17B60, 
17B62,
16W22} 

\keywords{Left-Alia algebras; Reflection groups; Yang-Baxter equations; Relative Rota-Baxter operators}

\maketitle


\tableofcontents

\allowdisplaybreaks

\section{Introduction}
A left-Alia algebra \cite{Dzh09} is a pair $(A,[\cdot,\cdot])$ consisting of a vector space $A$ and a bilinear map $[\cdot,\cdot]:A\otimes A\rightarrow A$ satisfying the symmetric Jacobi identity
	\begin{equation}\label{eq:0-Alia}
		[[x, y], z]+[[y, z], x]+[[z, x], y]=[[y, x], z]+[[z, y], x]+[[x, z], y],\;\forall x,y,z\in A.
	\end{equation}
There are various typical examples of left-Alia algebras, including Lie algebars, anti-pre-Lie algebras~\cite{LB2022} and mock-Lie algebras~\cite{Zusmanovich}.  Moreover, a class of non-Lie algebra examples can be realized from twisted derivations in invariant theory.

\subsection{Yang-Baxter equations and relative Rota-Baxter operators}
The classical Yang-Baxter equation (CYBE)  was  initially presented in tensor form as
$$
[r_{12},r_{13}]+[r_{12},r_{23}]+[r_{13},r_{23}]=0,
$$
where $r \in \g\otimes \g$ and
$(\frak g,[-,-])$ is a Lie algebra (see \cite{Chari} for more details). The CYBE serves as the ``semi-classical limit'' of the quantum Yang-Baxter equation and is closely related to classical integrable systems and quantum groups . Moreover, antisymmetric solutions of the CYBE give rise to (triangular) Lie bialgebras \cite{Chari,Dri}.
\delete{
Previous studies \cite{Belavin-D,Semenov,Kupershmidt} have explored the operator form  of the CYBE, which is an $\mathcal{O}$-operator  associated to the adjoint representation of $\g$.

It was shown that any $\mathcal{O}$-operators provide antisymmetric
solutions of the CYBE, while pre-Lie algebras naturally
provide $\mathcal{O}$-operators on the sub-adjacent Lie algebras \cite{Bai-O}.

Rota-Baxter operators can be viewed as an algebraic abstraction of the integral operator, see \cite{Guo-I} for more details.

Originally investigated by Baxter in his study of fluctuation theory in probability \cite{Baxter} and further developed by Rota \cite{Rota}.

 Rota-Baxter operators  appears in many areas such as quantum field theory renormalization \cite{Connes},  combinatorics \cite{Gu}, multiple zeta values in number theory \cite{Guo-Zhang} and algebraic operad \cite{Aguiar}.

The $\mathcal{O}$-operator associated to the representation of Lie algebras, also known as a Rota-Baxter operator of weight 0 on a Lie algebra \cite{Bai-Guo-R,Jiang}.

Furthermore,
Rota-Baxter operators are special cases of relative Rota-Baxter operators, which have inspired subsequent studies on deformations, cohomologies and homotopies of relative Rota-Baxter Lie algebras \cite{Sheng-R}, relative Rota-Baxter operators on Hopf algebras \cite{Li-Sheng} and bimodules over relative Rota-Baxter algebras \cite{Das}.
}

The notion of Rota-Baxter operators on associative algebras was originally investigated by Baxter in his study of fluctuation theory in probability \cite{Baxter} and further developed by Rota \cite{Rota}.
The study of Rota-Baxter operators has experienced a great expansion in recent years due to their broad connections with combinatorics, quantum field theory, number theory and operads (see \cite{Guo-I,Rota2} and the references therein). In particular, Rota-Baxter operators on Lie algebras arose independently as the operator forms of the CYBE. Explicitly, it was shown in \cite{Semenov} that an antisymmetric solution of the CYBE on a quadratic Lie algebra is equivalent to a linear map $P:\frak g\rightarrow\frak g$ satisfying
\begin{equation}\label{eq:rb0}
	[P(x),P(y)]=P\big([P(x),y]+[x,P(y)]\big),\;\forall x,y\in\frak g.
\end{equation}
The equation \eqref{eq:rb0} is just what we call nowadays a Rota-Baxter operator (of weight $0$) on a Lie algebra.
The operator form of the CYBE was later restudied in \cite{Kupershmidt}. Identifying an $r\in\frak g\otimes\frak g$ as a linear map $r^{\sharp}:\frak g^{*}\rightarrow \frak g$, Kupershmidt showed that when $r$ is antisymmetric, $r$ is a solution of  the CYBE if and only if $r^{\sharp}$ satisfies
\begin{equation}\label{eq:rrb}
	[r^{\sharp}(a^{*}),r^{\sharp}(b^{*})]=r^{\sharp}\Big( \mathrm{ad}^{*}\big(r^{\sharp}(a^{*})\big)b^{*}-\mathrm{ad}^{*}\big(r^{\sharp}(b^{*})\big)a^{*} \Big),\;\forall a^{*},b^{*}\in\frak g^{*}.
\end{equation}
Unifying \eqref{eq:rb0} and \eqref{eq:rrb}, the notion of a relative Rota-Baxter operator (originally named as an $\mathcal{O}$-operator  \cite{Kupershmidt}) of a Lie algebra $(\frak g,[-,-])$ associated to an arbitrary representation $(\rho,V)$ was also introduced as a natural generalization, which takes the form
\begin{equation}
	[T(u),T(v)]=T\Big( \rho\big(T(u)\big)v-\rho\big(T(v)\big)u \Big),\;\forall u,v\in V.
\end{equation}
Relative Rota-Baxter operators have inspired subsequent studies on deformations, cohomologies and homotopies of relative Rota-Baxter Lie algebras \cite{Sheng-R}, relative Rota-Baxter operators on Hopf algebras \cite{Li-Sheng} and bimodules over relative Rota-Baxter algebras \cite{Das}.

\subsection{Motivation}
A bialgebra structure is a vector space equipped with an algebra structure and a coalgebra structure satisfying some compatible conditions.
In addition to the aforementioned Lie bialgebras, some known examples are
antisymmetric infinitesimal bialgebras \cite{Bai2010,SW}, Leibniz bialgebras \cite{ShengT}, anti-pre-Lie bialgebras \cite{LB2023}, mock-Lie bialgebras \cite{Benali} and so on.
These bialgebras have a common property, that is, they have
equivalent characterizations in terms of Manin triples which
correspond to nondegenerate (symmetric or antisymmetric) bilinear forms on the algebra structures
satisfying certain invariant conditions. It is natural to consider the bialgebra theory for left-Alia algebras via the Manin triple approach.

In the previous article \cite{Kang-Liu}, we obtain a class of left-Alia algebras based on twisted derivations of commutative associative algebras.
The notions of  Manin triples of left-Alia algebras (with respect to the symmetric invariant bilinear form) and left-Alia bialgebras are also studied. Moreover, their equivalence  is figured out via specific matched pairs of left-Alia algebras.

\delete{
We also introduced the notions of a Manin triple of left-Alia algebras (with respect to the symmetric invariant bilinear form) and a left-Alia bialgebra, and showed that they are equivalent structures through matched pairs of left-Alia algebras with respect to the coadjoint representation.
However, the construction theory of left-Alia bialgebras remain unknown.}

In this paper, we aim to solve the problem on how to construct a left-Alia bialgebra. As an analogue of the studies on Lie bialgebras \cite{Chari,Dri}, we first introduce the notion of a left-Alia Yang-Baxter equation. Then we show that an antisymmetric solution of the left-Alia Yang-Baxter equation gives rise to a left-Alia bialgebra that we call triangular. Finally, we study solutions of the left-Alia Yang-Baxter equation in terms of the operator forms, namely the relative Rota-Baxter operators of left-Alia algebras.

\delete{
Motivated by the bialgebra theory involving Lie and other algebraic structures (see \cite{Dri,ShengT,SW,LB2022,LB2023,Benali}), we also established the equivalence between Manin triples  of left-Alia algebras and left-Alia bialgebras.
 Is it possible to define Yang-Baxter equation and relative Rota-Baxter operators for left-Alia algebras? In this paper, we first introduce the notion of  Yang-Baxter equations, triangular r-matrices and relative Rota-Baxter operators for left-Alia algebras. Based on the structure of left Alia bialgebras, we construct a class of triangular $r$-matrices for left-Alia algebras.}

\subsection{Outline of the paper}
This paper is organized as follows. In Section 2, we recall the notion of left-Alia algebras and their representations, including a class of non-trivial examples in invariant theory. In Section 3, we recall the notions of Manin triples of left-Alia algebras and left-Alia bialgebras. In Section 4, we introduce the concept of the left-Alia Yang-Baxter equation.
We show that antisymmetric solutions of the left-Alia Yang-Baxter equation give rise to triangular left-Alia bialgebras (see Proposition \ref{pro:4.2}).
In Section 5, we introduce relative Rota-Baxter operators of left-Alia algebras as the operator forms of solutions of the left-Alia Yang-Baxter equation. We show that in the antisymmetric condition, a solution of the left-Alia Yang-Baxter equation is one-to-one correspondence to a relative Rota-Baxter operator with respect to the coadjoint representation (Theorem \ref{thm:4.10}).
In Section 6, we introduce the notion of a pre-left-Alia algebra $(A,\succ,\prec)$,  which is the underlying algebraic structure of
a relative Rota-Baxter operator with respect to a representation $(l,r,V)$ of a left-Alia algebra. We prove that there is a
natural antisymmetric solution of the left-Alia Yang-Baxter equation in the left-Alia algebra $A\ltimes_{\mathcal{L}^{*}_{\succ},\mathcal{L}^{*}_{\succ}-\mathcal{R}^{*}_{\prec} }A^{*}$ (Theorem \ref{thm:4.16}).

Throughout this paper, unless otherwise specified, all vector spaces are finite-dimensional over an algebraically closed field $\mathbb{K}$ of characteristic zero, although many  notions and results remain valid in the infinite-dimensional case.

\section{Left-Alia algebras associated to invariant theory}
\subsection{Definition of left Alia algebras}
\begin{defi}\cite{Dzh09}
	A \textbf{{left-Alia algebra}} is a vector space $A$ together with a bilinear \linebreak  map $[\cdot,\cdot]:A\times A\rightarrow A$ satisfying the symmetric Jacobi property \eqref{eq:0-Alia}.
\end{defi}
\begin{rmk}
A left-Alia algebra $(A,[\cdot,\cdot])$ is a Lie algebra if and only if the bilinear map $[\cdot,\cdot]$ is antisymmetric.
\end{rmk}

\delete{
In~\cite{Dzh09}, there is also the notion of a right-Alia algebra, defined as a vector space $A$ together with a bilinear map $[\cdot,\cdot]':A\times A\rightarrow A$ satisfying
\begin{equation}\label{eq:right Alia}
		[x,[y,z]']'+[y,[z,x]']'+[z,[x,y]']'=[x,[z,y]']'+[y,[x,z]']'+[z,[y, x]']',\;\forall x,y,z\in A.
	\end{equation}
	It is clear that $(A,[\cdot,\cdot]')$ is a right-Alia algebra if and only if the opposite algebra $(A,[\cdot,\cdot])$ of $(A,[\cdot,\cdot]')$, given by
	$[x,y]=[y,x]'$, is a left-Alia algebra.
	Thus, our study on left-Alia algebras can straightforwardly generalize a parallel study on right-Alia algebras.
\begin{ex}\label{ex:3-left}
Let $A$ be a 3-dimensional vector space
over the complex field $\mathbb{C}$ with a basis $\{e_1, e_2, e_3\}$. Define a linear map $[\cdot,\cdot]:A\times A\rightarrow A$ by
\begin{align*}
&[e_1,e_2]=e_1,~~[e_1,e_3]=e_1,~~[e_2,e_1]=e_2,~~[e_3,e_1]=e_3,\\
   &[e_1,e_1]=[e_2,e_2]=[e_3,e_3]=[e_2,e_3]=[e_3,e_2]=e_1+e_2+e_3.
\end{align*}
Then $(A, [\cdot,\cdot])$ is a 3-dimensional left-Alia algebra.
\end{ex}
}

\begin{defi}
	Let $(A,\cdot)$ be a commutative associative algebra.
 A linear map $D:A\rightarrow A$ is called \textbf{a twisted derivation} respect to an $R:A\rightarrow A$ if $D$ satisfies the twisted Lebniz rule:
\begin{equation}\label{eq:s-der}
D(f\cdot g)=D(f)\cdot g+R(f)\cdot D(g),\ f,g\in A.
\end{equation}
\end{defi}
The following theorem in \cite{Kang-Liu} provide a class of examples of left-Alia algebras.
\begin{thm}\label{thm:R-D}\cite{Kang-Liu}
 Let $(A,\cdot)$ be a commutative associative algebra and $D$ be a twisted derivation. Define the bilinear map $[\cdot,\cdot]_R:A\times A\rightarrow A$ by
\begin{equation}\label{eq:R-der}
[x,y]_R:=x\cdot D(y)-R(y)\cdot D(x),\ x,y\in A.
\end{equation}
Then $(A,[\cdot,\cdot]_R)$ is a left-Alia algebra.
\end{thm}
A more general construction of left-Alia algebras is given in \cite{Dzh09}.
\begin{thm}\cite{Dzh09}\label{special}
Let $(A,\cdot)$ be a commutative associative algebra with linear maps $f,g:A\rightarrow A$.
	There is a left-Alia algebra $(A,[\cdot,\cdot])$ given by
	\begin{equation}\label{eq:fg}
		[x,y]=x\cdot f(y)+g(x\cdot y),\;\forall x,y\in A.
	\end{equation}
	which is called \textbf{a special left-Alia algebra} with respect to $(A,\cdot,f,g)$.
\end{thm}

\subsection{Examples of left-Alia algebras associated to invariant theory}
Let $G$ be a finite group and  $\mathbb {K}$ an algebraic closed field  of characteristic zero.
Suppose that $(\rho,V)$ is a $n$-dimensional faithful representation of $G$ and its dual representation is denoted by $(\rho,V^*)$.
Let $S=\mathbb {K}[V]=\mathbb {K}[x_1,\ldots,x_n]$ be the coordinate ring of $V$. Define a $G$-action on $S$ by
\begin{equation}\label{eq:action}
g\cdot \sum_{i_1,\ldots,i_n} k_{i_1,\ldots,i_n} x_1^{i_1}\ldots x_n^{i_n}:=\sum_{i_1,\ldots,i_n} k_{i_1,\ldots,i_n}(\rho (g)x_1)^{i_1}\ldots (\rho (g)x_n)^{i_n}.
\end{equation}
Define the ring of invariants by
\[S^{G}=\{f\in S: a\cdot f=f,\ \forall a\in G\}.\]
\begin{defi}\cite{IT}
A linear automorphism $R\in GL(V)$ is called \textbf{a pseudo-reflection} if $R^m=I$ for some $m\in \mathbb{N}^*$ and ${\rm Im}(I-R)$ is 1-dimensional.
\end{defi}
In invariant theory, Shephard, Todd and Chevalley figure out the following equivalent condition that $S^G$ is a polynomial algebra.
\begin{thm}\cite{Che,ST}
$S^G$ is a polynomial algebra if and only if $G\cong \rho(G)$ is generated by pseudo-reflections.
\end{thm}

For a pseudo-reflection $R$ on $V$, it naturally induces a pseudo-reflection on $V^*$ (also denoted by $R$).
For a fixed non-zero $l_R\in {\rm Im} (I-R)\subset V^*$, there exists a $\Delta_R\in V$ such that
\begin{equation}\label{eq:ref}
(I-R)x=\Delta_R(x)l_R,\ \ \forall x\in V^*.
\end{equation}

Denote also $R:S\rightarrow S$ to be an extension of  $R\in GL(V)$ satisfying
\[
R(k_1 f+k_2h)=k_1R(f)+k_2R(h)\ {\rm and}\  R(fh)=R(f)R(h).
\]
\begin{thm}\label{main12}\cite{IT}
For each $f\in S$, there exists a twisted derivation $D_R:S\rightarrow S$ with respect to $R$
such that
\begin{equation}
R(f)=f-D_R(f)l_R.
\end{equation}
\end{thm}
Motivated by Theorem \ref{main12} and Theorem \ref{thm:R-D}, we know that each pseodo-reflection $R$ on $V$ naturally induces a left-Alia algebra $(\mathbb{K}[V],[\cdot,\cdot]_R)$. The following is a concrete example.

\begin{ex}
	Let $R$ be a reflection defined by $R(x_1)=x_2,R(x_2)=x_1,R(x_3)=x_3$ on three-dimensional vector space $V^*$ with a basis $\{x_1,x_2,x_3\}$. On the coordinate ring \mbox{$S=\mathbb{K}[x_1,x_2,x_3]$}  of $V$, $R$ can be also denoted an extension of $R$ satisfying \mbox{$R(fg)=R(f)R(g)$} and $R(k_1 f+k_2h)=k_1R(f)+k_2R(h)$. Let $D$ be the twisted derivation on $S$ induced by the reflection $R$. It follows from Theorem~\ref{main12} that $R(f)=f-D(f)(x_1-x_2)$. Take two polynomials in $S$, denoted by $f=\underset{i}{\sum}k_i f_i,g=\underset{j}{\sum}h_j g_j,$ where $f_i=x_1^{n_{i,1}}x_2^{n_{i,2}}x_3^{n_{i,3}}, g_j=x_1^{m_{j,1}}x_2^{m_{j,2}}x_3^{m_{j,3}}$ are basis of $S$ and $k_i,h_j\in\mathbb{K}$. We have
	\begin{align*}
		D(x_1)=&1, D(x_2)=-1, D(x_3)=0,\\
		D(x_1^{n_1})=&x_1^{n_1-1}+x_1^{n_1-2}x_2+\cdots+x_2^{n_1-1},\\
		D(x_2^{n_2})=&-x_1^{n_2-1}-x_1^{n_2-1}x_2-\cdots-x_2^{n_2-1},\\
		D(x_3^{n_3})=&0,
	\end{align*}
and
\begin{align*}
		D(\underset{i}{\sum} k_ix_1^{n_{i,1}}x_2^{n_{i,2}}x_3^{n_{i,3}})=&\underset{i}{\sum}k_i(D(x_1^{n_{i,1}}x_2^{n_{i,2}})x_3^{n_{i,3}}+R(x_1^{n_{i,1}}x_2^{n_{i,2}})D(x_3^{n_{i,3}}))\\
		=&\underset{i}{\sum} k_i(x_1^{n_{i,1}}D(x_2^{n_{i,2}})+R(x_2^{n_{i,2}})D(x_1^{n_{i,1}}))x_3^{n_{i,3}}\\
		=&\underset{i}{\sum} k_i(x_1^{n_{i,1}+n_{i,2}-1}+x_1^{n_{i,1}+n_{i,2}-2}x_2+\cdots+x_1^{n_{i,2}}x_2^{n_{i,1}-1}\\
		&-x_1^{n_{i,1}+n_{i,2}-1}-\cdots-x_1^{n_1}x_2^{n_2-1})x_3^{n_{i,3}}.
	\end{align*}
	
Moreover, the bilinear map
	\begin{align*}
		[f,g]_R=&\underset{i,j}{\sum}k_ih_j [f_i,g_j]_R\\
		=&\underset{i,j}{\sum}k_ih_j (f_iD(g_j)-R(g_j)D(f_i))\\
		=&\underset{i,j}{\sum}k_ih_j (x_1^{n_{i,1}+m_{j,1}+m_{j,2}-1}x_2^{n_{i,2}}+\cdots+x_1^{n_{i,1}+m_{j,2}}x_2^{n_{i,2}+m_{j,1}-1}\\
		&-x_1^{n_{i,1}+m_{j,1}+m_{j,2}-1}x_2^{n_{i,2}}-\cdots-x_1^{m_{j,1}+n_{i,1}}x_2^{n_{i,2}+m_{j,2}-1}\\
		&-x_1^{m_{j,2}+n_{i,1}+n_{i,2}-1}x_2^{m_{j,1}}-\cdots-x_1^{n_{i,2}+m_{j,2}}x_2^{n_{i,1}+m_{j,1}-1}\\
		&+x_1^{m_{j,2}+n_{i,1}+n_{i,2}-1}x_2^{m_{j,1}}+\cdots+x_1^{n_{i,1}+m_{j,2}}x_2^{n_{i,2}+m_{j,1}-1})x_3^{m_{j,3}+n_{i,3}}
	\end{align*}
makes $(S,[\cdot,\cdot]_R)$ into a left-Alia algebra.
\end{ex}

\subsection{Representation of left-Alia algebras}
In this subsection, we recall the representation theory \cite{Kang-Liu} of left-Alia algebras.

\vspace{6pt}
\begin{defi}
	A \textbf{{representation}} of a left-Alia algebra $(A,[\cdot,\cdot])$ is a triple $(l,r,V)$, where $V$ is a vector space and $l,r:A\rightarrow\mathrm{End}(V)$ are linear maps such that the following equation holds:
	\begin{equation}\label{eq:rep 0-Alia}
		l([x,y])-l([y, x])= r(x)r(y)v-r(y)r(x)+r(y)l(x)v-r(x)l(y),\;\forall x,y\in A, v\in V.
	\end{equation}
	
\end{defi}

%

\begin{pro}
Let $(A,[\cdot,\cdot])$ be a left-Alia algebra, $V$ be a vector space and $l,r:A\rightarrow\mathrm{End}(V)$ be linear maps.
Then $(l,r,V)$ is a representation of $(A,[\cdot,\cdot])$ if and only if there is a left-Alia algebra on the direct sum $ A\oplus V$ of vector spaces given by
\begin{equation}\label{eq:semi-d}
	[x+u,y+v]_d= [x, y]+l(x)v+r(y)u ,\;\forall x,y\in A, u,v\in V.
\end{equation}
In this case, we denote $(A\oplus V,[\cdot,\cdot]_d)=A\ltimes_{l,r}V$.
\end{pro}

\delete{
An equivalent characterisation of representations on left-Alia algebras is
given in the following.

\begin{pro}
	Let $(A,[\cdot,\cdot])$ be a left-Alia algebra, $V$ be a vector space and $l,r:A\rightarrow\mathrm{End}(V)$ be linear maps.
	Then, $(l,r,V)$ is a representation of $(A,[\cdot,\cdot])$ if and only if there is a left-Alia algebra on the direct sum $d=A\oplus V$ of vector spaces given by
	\begin{equation}\label{eq:semi-d}
		[x+u,y+v]_d= [x, y]+l(x)v+r(y)u ,\;\forall x,y\in A, u,v\in V.
	\end{equation}
	In this case, we denote $(A\oplus V,[\cdot,\cdot]_d)=A\ltimes_{l,r}V$.
\end{pro}

For a vector space $A$ with a bilinear map $[\cdot,\cdot]:A\times A\rightarrow A$, in the case $V=A$, set linear maps $\mathcal{L}_{[\cdot,\cdot]},\mathcal{R}_{[\cdot,\cdot]}:A\rightarrow\mathrm{End}(A)$ using
\begin{equation*}
	\mathcal{L}_{[\cdot,\cdot]}(x)y=[x, y]=\mathcal{R}_{[\cdot,\cdot]}(y)x,\;\forall x,y\in A.
\end{equation*}
}

\begin{ex}
	Let $(A,[\cdot,\cdot])$ be a left-Alia algebra.
	Set linear maps $\mathcal{L}_{[\cdot,\cdot]},\mathcal{R}_{[\cdot,\cdot]}:A\rightarrow\mathrm{End}(A)$ by
	\begin{equation*}
		\mathcal{L}_{[\cdot,\cdot]}(x)y=[x, y]=\mathcal{R}_{[\cdot,\cdot]}(y)x,\;\forall x,y\in A.
	\end{equation*}
	Then, $(\mathcal{L}_{[\cdot,\cdot]},\mathcal{R}_{[\cdot,\cdot]},A)$ is a representation of $(A,[\cdot,\cdot])$, which is called an \textbf{{adjoint representation}}.
\end{ex}


Let $A$ and $V$ be vector spaces. For a linear map $l :A\rightarrow\mathrm{End}(V)$, define a linear map $l^{*}:A\rightarrow\mathrm{End}(V^{*})$ by
\begin{equation*}
	\langle l^{*}(x)u^{*},v\rangle=-\langle u^{*},l(x)v\rangle,\;\forall x\in A, u^{*}\in V^{*}, v\in V.
\end{equation*}

\begin{pro}\label{pro:dual rep}\cite{Kang-Liu}
	Let $(l,r,V)$ be a representation of a left-Alia algebra $(A,[\cdot,\cdot])$. Then,  $( l^{*}, l^{*}-r^{*}, V^{*} )$ is also a representation of $(A,[\cdot,\cdot])$. In particular, $( \mathcal{L}^{*}_{[\cdot,\cdot]},\mathcal{L}^{*}_{[\cdot,\cdot]}-\mathcal{R}^{*}_{[\cdot,\cdot]}, A^{*} )$ is a representation of $(A,[\cdot,\cdot])$, which is called the \textbf{{coadjoint representation}}.
\end{pro}

\delete{
\begin{rmk}
$( \mathcal{L}^{*}_{[\cdot,\cdot]},\mathcal{L}^{*}_{[\cdot,\cdot]}-\mathcal{R}^{*}_{[\cdot,\cdot]}, A^{*} )$ is a representation of $(A,[\cdot,\cdot])$, which is called the \textbf{{coadjoint representation}}.
\end{rmk}}

%

%

\delete{
	\begin{Remark}
		Let $\g$ be a Lie algebra and $\g^*$ be its dual vector space. If a linear map $\ad^{*}:\g\rightarrow\mathrm{End}(\g^*)$ satisfying
		\begin{equation}\label{eq:ad}
			\langle \ad^{*}(x)\xi,y\rangle=-\langle \xi,\ad(x)y\rangle,\;\forall x,y\in A, \xi\in \g^{*}.
		\end{equation}
		Then $\ad^*$ is a representation of $\g$ in $\g^*$, which is called the coadjoint
		representation of $\g$. In Proposition~\ref{pro:dual rep}, if the left-Alia algebra $(A,[\cdot,\cdot])$ is antisymmetric, that is for all $x\in A$ $(\mathcal{L}_{[\cdot,\cdot]}-\mathcal{R}_{[\cdot,\cdot]})(x)=2\mathcal{L}_{[\cdot,\cdot]}(x)$, then $\mathcal{L}_{[\cdot,\cdot]}-\mathcal{R}_{[\cdot,\cdot]}$ is a adjoint
		representation of Lie algebra. Furthermore, by \eqref{eq:ad}, $\mathcal{L}_{[\cdot,\cdot]}^*$ and
		$\mathcal{L}^{*}_{[\cdot,\cdot]}-\mathcal{R}^{*}_{[\cdot,\cdot]}$ are coadjoint
		representations of Lie algebras.
\end{Remark}}

\section{Left-Alia bialgebras}
In this section, we recall some basic results on left-Alia bialgebras \cite{Kang-Liu}.
\begin{defi}\label{qua}
	A \textbf{{quadratic left-Alia algebra}} is a triple $(A,[\cdot,\cdot],\mathcal{B})$, where $(A,[\cdot,\cdot])$ is a left-Alia algebra and $\mathcal{B}$ is a nondegenerate symmetric bilinear form on $A$ which is invariant in the sense that
	
	\vspace{-6pt}\begin{equation}\label{eq:quad}
		\mathcal{B}([x, y],z)=\mathcal{B}(x,[z, y]-[y, z]),\;\forall x,y,z\in A.
	\end{equation}
\end{defi}
	
\delete{
\begin{rmk}
The symmetric property of $\mathcal{B}$ implies that
\begin{equation*}\label{eq:inv}
    \mathcal{B}([x, y],z)+\mathcal{B}(y,[x, z])=0,\;\forall x,y,z\in A.
\end{equation*}
\end{rmk}}



There is a typical class of quadratic left-Alia algebras realized from Theorem \ref{special}.

\begin{pro}\label{pro:2.8}
	Let $(A,\cdot)$ be a commutative associative algebra and $f:A\rightarrow A$ be a linear map.
	Let $\mathcal{B}$ be a nondegenerate symmetric invariant bilinear form on $(A,\cdot)$
	and $\hat{f}:A\rightarrow A$ be the adjoint map of $f$ with respect to $\mathcal{B}$, given by
	
	\vspace{-6pt}\begin{equation*}
		\mathcal{B}\big(\hat{f}(x),y\big)=\mathcal{B}\big(x,f(y)\big),\;\forall x,y\in A.
	\end{equation*}
	Then there is a quadratic left-Alia algebra $(A,[\cdot,\cdot],\mathcal{B})$, where $(A,[\cdot,\cdot])$ is the special left-Alia algebra with respect to $(A,\cdot,f,-\hat f)$, that is,
	
	\vspace{-6pt}\begin{equation}\label{eq:hat}
		[x, y]=x\cdot f(y)-\hat{f}(x\cdot y).
	\end{equation}
\end{pro}

\delete{
\begin{ex}\label{ex:2.9}
Let $(A,[\cdot,\cdot])$ be a left-Alia algebra and $(\mathcal{L}_{[\cdot,\cdot]},\mathcal{R}_{[\cdot,\cdot]},A)$ be the adjoint representation of $(A,[\cdot,\cdot])$.
By Proposition \ref{pro:semi} and Proposition \ref{pro:dual rep}, there is a left-Alia algebra $A\ltimes_{\mathcal{L}^{*}_{[\cdot,\cdot]}, \mathcal{L}^{*}_{[\cdot,\cdot]}-\mathcal{R}^{*}_{[\cdot,\cdot]}}A^{*}$ on $d=A\oplus A^*$  given by \eqref{eq:semi-d}.
There is a natural nondegenerate symmetric bilinear form $\mathcal{B}_{d}$ on $A\oplus A^{*}$ given by
\begin{equation}\label{eq:Bd}
    \mathcal{B}_{d}(x+a^{*},y+b^{*})=\langle x, b^{*}\rangle+\langle a^{*},y\rangle,\;\forall x,y\in A, a^{*},b^{*}\in A^{*}.
\end{equation}
For all $x,y,z\in A, a^{*},b^{*},c^{*}\in A^{*}$, we have
\begin{eqnarray*}
\mathcal{B}_{d}\big([(x+a^{*}),(y+b^{*})]_d,z+c^{*})
&=&\mathcal{B}_{d}\big([x, y]+\mathcal{L}^{*}_{[\cdot,\cdot]}(x)b^{*}+(\mathcal{L}^{*}_{[\cdot,\cdot]}-\mathcal{R}^{*}_{[\cdot,\cdot]})(y)a^{*},z+c^{*}\big)\\
&=&\langle [x, y],c^{*}\rangle+\langle \mathcal{L}^{*}_{[\cdot,\cdot]}(x)b^{*}+(\mathcal{L}^{*}_{[\cdot,\cdot]}-\mathcal{R}^{*}_{[\cdot,\cdot]})(y)a^{*},z\rangle\\
&=&\langle [x,y],c^{*}\rangle-\langle [x,z], b^{*}\rangle+\langle a^{*},[z,y]-[y,z]\rangle,\\
 \mathcal{B}_{d}\big(x+a^{*},[z+c^{*},x+b^{*}]_d\big)&=&\langle [z,y], a^{*}\rangle-\langle [z,x], b^{*}\rangle+\langle c^{*}, [y,x]-[x, y]\rangle,\\
 \mathcal{B}_{d}\big(x+a^{*},[y+b^{*},z+c^{*}]_d\big)&=&\langle [y,z], a^{*}\rangle-\langle [y,x], c^{*}\rangle+\langle b^{*}, [z,x]-[x,z]\rangle.
\end{eqnarray*}
Hence we have
\begin{equation*}
\mathcal{B}_{d}\big( [x+a^{*},y+b^{*}]_d,z+c^{*} \big)=\mathcal{B}_{d}\big( x+a^{*}, [z+c^{*},y+b^{*}]_d- [y+b^{*},z+c^{*}]_d\big),
\end{equation*}
and thus $(A\ltimes_{\mathcal{L}^{*}_{[\cdot,\cdot]}, \mathcal{L}^{*}_{[\cdot,\cdot]}-\mathcal{R}^{*}_{[\cdot,\cdot]}}A^{*},\mathcal{B}_{d})$ is a quadratic left-Alia algebra.
\end{ex}
}

\begin{defi}\label{Manin}
Let $(A,[\cdot,\cdot]_A )$ and $(A^{*},[\cdot,\cdot]_{A^{*}} )$ be left-Alia algebras.
Assume  that there is a left-Alia algebra structure $(A \oplus A^{*},[\cdot,\cdot]_d )$ on $A\oplus A^{*}$  which contains $(A,[\cdot,\cdot]_A )$ and $(A^{*},[\cdot,\cdot]_{A^{*}} )$ as left-Alia subalgebras.
Suppose that  the natural nondegenerate symmetric bilinear form $\mathcal{B}_{d}$ given by
\begin{equation}\label{eq:Bd}
	\mathcal{B}_{d}(x+a^{*},y+b^{*})=\langle x,b^{*}\rangle+\langle a^{*},y\rangle,\;\forall x,y\in A, a^{*},b^{*}\in A^{*}.
\end{equation}
  is invariant on $(A \oplus A^{*},[\cdot,\cdot]_d )$, that is, $(A \oplus A^{*},[\cdot,\cdot]_{d},\mathcal{B}_{d} )$ is a quadratic left-Alia algebra.
Then we say $\big( (A\oplus A^{*},[\cdot,\cdot]_{d},\mathcal{B}_{d}), A, A^{*} \big)$ is a {\bf Manin triple of left-Alia algebras (with respect to the symmetric invariant bilinear form)}.
\end{defi}
\begin{defi}\label{defi:anti-pre-Lie coalgebras}
	A  \textbf{{left-Alia coalgebra}} is a pair $(A,\delta)$, such that $A$ is a vector space and 
	$\delta:A\rightarrow A\otimes A$ is a co-multiplication satisfying
	\begin{equation}\label{eq:defi:coalgebra}
		(\mathrm{id}^{\otimes 3}+\xi+\xi^{2})(\tau\otimes \mathrm{id}-\mathrm{id}^{\otimes 3})(\delta\otimes \mathrm{id})\delta=0,
	\end{equation}
	where $\tau(x\otimes y)=y\otimes x$ and $\xi(x\otimes y\otimes z)=y\otimes z\otimes x$ for all $x,y,z\in A$.
\end{defi}
\begin{pro}\label{pro:coalgebras}
	Let $A$ be a vector space and $\delta:A\rightarrow A\otimes A$ be a co-multiplication.
	Let $[\cdot,\cdot]_{A^*} :A^{*}\otimes A^{*}\rightarrow A^{*}$ be the linear dual of $\delta$, that is,
	\begin{equation}\label{db}
		\langle [a^{*}, b^{*}]_{A^*},x\rangle=\langle\delta^{*}(a^{*}\otimes b^{*}),x\rangle=\langle a^{*}\otimes b^{*},\delta(x)\rangle, \;\;\forall a^{*},b^{*}\in A^{*}, x\in A.
	\end{equation}
	Then $(A,\delta)$ is a left-Alia coalgebra if and only if $(A^{*},[\cdot,\cdot]_{A^{*}} )$ is a left-Alia algebra.
\end{pro}

\begin{defi}\label{bialgebra}
	A \textbf{{left-Alia bialgebra}} is a triple $(A,[\cdot,\cdot],\delta)$, such that
	$(A,[\cdot,\cdot])$ is  a left-Alia algebra, $(A,\delta)$ is a left-Alia coalgebra and the following equation holds:
	\begin{equation}\label{eq:bialg}
		(\tau-\mathrm{id}^{2})\big( \delta([x,y]-[y,x])+(\mathcal{R}_{[\cdot,\cdot]}(x)\otimes\mathrm{id})\delta(y)-(\mathcal{R}_{[\cdot,\cdot]}(y)\otimes\mathrm{id})\delta(x) \big)=0,\;\forall x,y\in A.
	\end{equation}
\end{defi}
The following theorem figures out the one-one correspondence between left-Alia bialgebras and Manin triples of left-Alia algebras.
\begin{thm}
 Let $(A,[\cdot,\cdot]_A )$ be a left-Alia algebra. Suppose that there is a left-Alia algebra structure $(A^{*},[\cdot,\cdot]_{A^{*}} )$ on the dual space $A^{*}$, and $\delta:A\rightarrow A\otimes A$ is the linear dual of $[\cdot,\cdot]_{A^*}$.
 Then $(A,[\cdot,\cdot]_A,\delta)$ is a left-Alia bialgebra
 if and only if  there is a Manin triple of left-Alia algebras $\big( (A\oplus A^{*},[\cdot,\cdot]_{d},\mathcal{B}_{d}), A, A^{*} \big)$.
 In this case, the multiplication $[-,-]_{d}$ on $A\oplus A^{*}$ is given by
 \begin{eqnarray*}
 [x+a^{*},y+b^{*}]_d&=&[x, y]_A+\mathcal{L}^{*}_{[\cdot,\cdot]_{A^*}}(a^{*})y+(\mathcal{L}^{*}_{[\cdot,\cdot]_{A^*}}-\mathcal{R}^{*}_{[\cdot,\cdot]_{A^*}})(b^{*})x\nonumber\\
 &&+[a^{*}, b^{*}]_{A^*}+\mathcal{L}^{*}_{[\cdot,\cdot]_A}(x)b^{*}+(\mathcal{L}^{*}_{[\cdot,\cdot]_{A }}-\mathcal{R}^{*}_{[\cdot,\cdot]_{A }})(y)a^{*},\label{eq:mp dual rep}
 \end{eqnarray*}
 for all $x,y\in A, a^{*},b^{*}\in A^{*}$.
\end{thm}
  \delete{
 		\begin{align}
		[x+a^{*},y+b^{*}]_d&=[x, y]_A+\mathcal{L}^{*}_{[\cdot,\cdot]_{A^*}}(a^{*})y+(\mathcal{L}^{*}_{[\cdot,\cdot]_{A^*}}-\mathcal{R}^{*}_{[\cdot,\cdot]_{A^*}})(b^{*})x\nonumber\\
		&+[a^{*}, b^{*}]_{A^*}+\mathcal{L}^{*}_{[\cdot,\cdot]_A}(x)b^{*}+(\mathcal{L}^{*}_{[\cdot,\cdot]_{A }}-\mathcal{R}^{*}_{[\cdot,\cdot]_{A }})(y)a^{*},\label{eq:mp dual rep}
	\end{align}
 and $\mathcal{B}_{d}$ is defined by (\ref{eq:Bd}).
 Then, $(A,[\cdot,\cdot]_A,\delta)$ is a left-Alia bialgebra
 if and only if $\big( d=(A\oplus A^{*},[\cdot,\cdot]_{d},\mathcal{B}_{d}), A, A^{*} \big)$ is a Manin triple.}

\delete{
\begin{ex}
	Let $(A,[\cdot,\cdot]_{A})$ be the three-dimensional left-Alia algebra given in Example~\ref{ex:3-left}.
	Then, there is a left-Alia bialgebra $(A,[\cdot,\cdot]_{A},\delta)$ with a non-zero co-multiplication $\delta$ on $A$, given by
	\begin{equation*}
		\delta(e_{1})=e_{1}\otimes e_{1}.
	\end{equation*}
	Then,  there is a Manin triple $\big( (A\oplus A^{*},[\cdot,\cdot],\mathcal{B}_{d}), A, A^{*}\big)$. 
where
	\begin{equation*}
		[e^{*}_{1},e^{*}_{1}]_{A^{*}}=e^{*}_{1},
	\end{equation*}
	and the multiplication $[\cdot,\cdot]$ on $A\oplus A^{*}$ is given by \eqref{eq:mp dual rep}.
\end{ex}
}

\section{Triangular left-Alia bialgebras and the left-Alia Yang-Baxter equation}

In this section, we introduce the notion of the left-Alia Yang-Baxter equation in a left-Alia algebra.
An antisymmetric solution of the left-Alia Yang-Baxter equation gives rise to a left-Alia bialgebra which we call triangular.


\begin{defi}
Let $(A,[\cdot,\cdot])$ be a left-Alia algebra and $r=\sum\limits_{i}u_{i}\otimes v_{i}\in A\otimes A$.
Set
\begin{equation}
    \mathrm{Al}(r)=\sum_{i,j}[u_{i},u_{j}]\otimes v_{i}\otimes v_{j}+u_{i}\otimes( [u_{j}, v_{i}]-[v_{i}, u_{j}] )\otimes v_{j}-u_{i}\otimes u_{j}\otimes [v_{j}, v_{i}]  \in A\otimes A\otimes A.
\end{equation}
We say $r$ is a solution of the {\bf left-Alia Yang-Baxter equation} if $\mathrm{Al}(r)=0$.
\end{defi}

\begin{pro}\label{pro:4.2}
Let $(A,[\cdot,\cdot])$ be a left-Alia algebra and $r=\sum\limits_{i}u_{i}\otimes v_{i}\in A\otimes A$.
Let $h:A\rightarrow\mathrm{End}(A\otimes A)$ be a linear map given by
\begin{equation}
	h(x)= (\mathcal{R}_{[\cdot,\cdot]}-\mathcal{L}_{[\cdot,\cdot]})(x)\otimes\mathrm{id}-\mathrm{id}\otimes\mathcal{R}_{[\cdot,\cdot]}(x),\;\forall x\in A.
\end{equation}
Set a co-multiplication $\delta_{r}:A \rightarrow A\otimes A$ by
\begin{equation}\label{eq:co}
   \delta_{r}(x)=h(x)r=\big((\mathcal{R}_{[\cdot,\cdot]}-\mathcal{L}_{[\cdot,\cdot]})(x)\otimes\mathrm{id}-\mathrm{id}\otimes\mathcal{R}_{[\cdot,\cdot]}(x)\big)r,\;\forall x\in A.
\end{equation}
\begin{enumerate}
    \item For all $x\in A$, \eqref{eq:defi:coalgebra} holds if and only if
\begin{eqnarray}
&&(\mathrm{id}+\xi+\xi^{2})\bigg( \big(\mathrm{id}\otimes\mathrm{id}\otimes\mathcal{R}_{[\cdot,\cdot]}(x)\big) (\mathrm{id}\otimes\tau)\mathrm{Al}(r)+
\sum_{j} h([u_{j}, x]-[x, u_{j}])\big(r+\tau(r)\big)\otimes v_{j}\notag\\
&&+\tau h(u_{j})\big( r+\tau(r)\big) \otimes v_{j}
+ v_{j}\otimes\big(\mathcal{R}_{[\cdot,\cdot]}(u_{j})\otimes\mathrm{id}\big)h(x)\big(r+\tau(r)\big)\notag\\
&&+\big(\mathrm{id}\otimes\mathcal{R}_{[\cdot,\cdot]}(v_{j})\otimes\mathcal{R}_{[\cdot,\cdot]}(x)\big)\Big(u_{j}\otimes\big(r+\tau(r)\big)\Big)\nonumber\\
&&+
\big(\mathrm{id}\otimes\mathcal{R}_{[\cdot,\cdot]}(u_{j})\otimes\mathcal{R}_{[\cdot,\cdot]}(x)\big)\Big(v_{j}\otimes\big(r+\tau(r)\big)\Big)\Bigg)=0.\ \ \ \ \ \ \label{eq:com1}
\end{eqnarray}
\item For all $x,y\in A$, \eqref{eq:bialg} holds if and only if
\begin{equation}\label{eq:com2}
\tau h([x, y]-[y, x])\big(r+\tau(r)\big)-\big(\mathrm{id}\otimes\mathcal{R}_{[\cdot,\cdot]}(y)\big)\tau h(x)\big( r+\tau(r)\big)
+\big(\mathrm{id}\otimes\mathcal{R}_{[\cdot,\cdot]}(x)\big)\tau h(y)\big( r+\tau(r)\big)=0.
\end{equation}
\end{enumerate}
\end{pro}
\begin{proof}
    {\bf (a)} For all $x\in A$, we have
    \begin{eqnarray*}
&&(\delta_{r}\otimes\mathrm{id})\delta_{r}(x)=(\delta_{r}\otimes\mathrm{id})(\sum_{i} [u_{i}, x]\otimes v_{i}-[x, u_{i}]\otimes v_{i}-u_{i}\otimes [v_{i}, x])\\
&&=\sum_{i,j} [u_{j},[u_{i}, x]]\otimes v_{j}\otimes v_{i}-[[u_{i},x], u_{j}]\otimes v_{j}\otimes v_{i}-u_{j}\otimes [v_{j},[u_{i},x]]\otimes v_{i}\\
&&\ \ -[u_{j},[x, u_{i}]]\otimes v_{j}\otimes v_{i}+[[x,u_{i}],u_{j}]\otimes v_{j}\otimes v_{i}+u_{j}\otimes [v_{j},[x,u_{i}]]\otimes v_{i}\\
&&\ \ -[u_{j}, u_{i}]\otimes v_{j}\otimes [v_{i},x]+[u_{i}, u_{j}]\otimes v_{j}\otimes [v_{i}, x]+u_{j}\otimes [v_{j}, u_{i}]\otimes [v_{i}, x].
    \end{eqnarray*}
Then
    \begin{eqnarray*}
&& (\tau\otimes\mathrm{id}-\mathrm{id}^{\otimes 3})(\delta_{r}\otimes\mathrm{id})\delta_{r}(x)\\
&&=\sum_{i,j} v_{j}\otimes [u_{j}, [u_{i},x]]\otimes v_{i}-v_{j}\otimes [[u_{i}, x], u_{j}]\otimes v_{i}-[v_{j},[u_{i},x]]\otimes u_{j}\otimes v_{i}\\
&&\ \ -v_{j}\otimes [u_{j},[x,u_{i}]]\otimes v_{i}+v_{j}\otimes [[x\, u_{i}], u_{j}]\otimes v_{i}+[v_{j},[x,u_{i}]]\otimes u_{j}\otimes v_{i}\\
&&\ \ -v_{j}\otimes [u_{j},u_{i}]\otimes [v_{i},x]+v_{j}\otimes [u_{i},u_{j}]\otimes [v_{i}, x]+[v_{j},u_{i}]\otimes u_{j}\otimes [v_{i},x]\\
&&\ \ -[u_{j},[u_{i}, x]]\otimes v_{j}\otimes v_{i}+[[u_{i},x], u_{j}]\otimes v_{j}\otimes v_{i}+u_{j}\otimes [v_{j},[u_{i},x]]\otimes v_{i}\\
&&\ \ +[u_{j},[x,u_{i}]]\otimes v_{j}\otimes v_{i}-[[x,u_{i}],u_{j}]\otimes v_{j}\otimes v_{i}-u_{j}\otimes [v_{j},[x, u_{i}]]\otimes v_{i}\\
&&\ \ +[u_{j}, u_{i}]\otimes v_{j}\otimes [v_{i}, x]-[u_{i}, u_{j}]\otimes v_{j}\otimes [v_{i},x]-u_{j}\otimes [v_{j},u_{i}]\otimes [v_{i},x].
    \end{eqnarray*}
We have
\begin{eqnarray*}
&&\sum_{i,j} -v_{j}\otimes [u_{j},u_{i}]\otimes [v_{i},x]+v_{j}\otimes [u_{i},u_{j}]\otimes [v_{i}, x]+[v_{j},u_{i}]\otimes u_{j}\otimes [v_{i},x]\\
&&\ \ +[u_{j},u_{i}]\otimes v_{j}\otimes [v_{i}, x]-[u_{i}, u_{j}]\otimes v_{j}\otimes [v_{i}, x]-u_{j}\otimes [v_{j},u_{i}]\otimes [v_{i}, x]\\
&&=-\sum_{j}\tau h(u_{j})\big( r+\tau(r)\big) \otimes [v_{j}, x]-\sum_{i,j}\big([u_{i}, u_{j}]\otimes v_{j}\otimes [v_{i},x]+u_{i}\otimes [u_{j}, v_{i}]\otimes [v_{j}, x]\big),
\end{eqnarray*}

\begin{eqnarray*}
&&\sum_{i,j} v_{j}\otimes [u_{j}, [u_{i}, x]]\otimes v_{i}-v_{j}\otimes [[u_{i},x], u_{j}]\otimes v_{i}-[v_{j},[u_{i}, x]]\otimes u_{j}\otimes v_{i}\\
&&\ \ -[u_{j},[u_{i}, x]]\otimes v_{j}\otimes v_{i}+[[u_{i},x],u_{j}]\otimes v_{j}\otimes v_{i}+u_{j}\otimes [v_{j},[u_{i},x]]\otimes v_{i}\\
&&=-\sum_{j}h([u_{j}, x])\big(r+\tau(r)\big)\otimes v_{j}-\sum_{i,j}\big( v_{j}\otimes[[u_{i}, x], u_{j}]\otimes v_{i}+[[u_{i},x], v_{j}]\otimes u_{j}\otimes v_{i} \big),
\end{eqnarray*}

\begin{eqnarray*}
&&\sum_{i,j} -v_{j}\otimes [u_{j},[x,u_{i}]]\otimes v_{i}+v_{j}\otimes [[x,u_{i}], u_{j}]\otimes v_{i}+[v_{j},[x,u_{i}]]\otimes u_{j}\otimes v_{i}\\
&&\ \ +[u_{j},[x, u_{i}]]\otimes v_{j}\otimes v_{i}-[[x,u_{i}], u_{j}]\otimes v_{j}\otimes v_{i}-u_{j}\otimes [v_{j},[x, u_{i}]]\otimes v_{i}\\
&&=\sum_{j} h([x,u_{j}])\big( r+\tau(r)\big)\otimes v_{j}+\sum_{i,j} v_{j}\otimes[[x, u_{i}], u_{j}]\otimes v_{i}+[[x, u_{i}], v_{j}]\otimes u_{j}\otimes v_{i},\\
&&\sum_{i,j} v_{j}\otimes[[x, u_{i}], u_{j}]\otimes v_{i}-v_{j}\otimes [[u_{i}, x], u_{j}]\otimes v_{i}\\
&&=-\sum_{i,j} v_{j}\otimes\big(\mathcal{R}(u_{j})\otimes\mathrm{id}\big)h(x)\big(r+\tau(r)\big)+\sum_{i,j}-v_{j}\otimes [u_{i}, u_{j}]\otimes [v_{i}, x]\\
&&+v_{j}\otimes[[v_{i}, x], u_{j}]\otimes u_{i}-v_{j}\otimes[[x,v_{i}],u_{j}]\otimes u_{i}-v_{j}\otimes [v_{i},u_{j}]\otimes [u_{i}, x].
\end{eqnarray*}
By the action of $\xi$ we have
\begin{eqnarray*}
&&(\mathrm{id}+\xi+\xi^{2})\big(v_{j}\otimes [[v_{i}, x], u_{j}]\otimes u_{i}-v_{j}\otimes [[x, v_{i}], u_{j}]\otimes u_{i}\\
&&-[[u_{i}, x], v_{j}]\otimes u_{j}\otimes v_{i}+[[x, u_{i}], v_{j}]\otimes u_{j}\otimes v_{i}\big)\\
&&=(\mathrm{id}+\xi+\xi^{2})\big(v_{j}\otimes [[v_{i}, x], u_{j}]\otimes u_{i}-v_{j}\otimes [[x, v_{i}], u_{j}]\otimes u_{i}\\
&&\ \ -v_{j}\otimes[[u_{j}, x], v_{i}]\otimes u_{i}+v_{j}\otimes[[x, u_{j}], v_{i}]\otimes u_{i} \big)\\
&&\overset{\eqref{eq:0-Alia}}{=}(\mathrm{id}+\xi+\xi^{2})\big( v_{j}\otimes[[v_{i}, u_{j}], x]\otimes u_{i}-v_{j}\otimes [[u_{j}, v_{i}], x]\otimes u_{i}\big).
\end{eqnarray*}
Hence we have
\begin{eqnarray*}
&&(\mathrm{id}+\xi+\xi^{2})\big( \sum_{i,j}-v_{j}\otimes [u_{i}, u_{j}]\otimes [v_{i}, x]-v_{j}\otimes [v_{i}, u_{j}]\otimes [u_{i}, x]+u_{i}\otimes v_{j}\otimes [[v_{i}, u_{j}], x]\\
&&-u_{i}\otimes v_{j}\otimes [[u_{j}, v_{i}], x]-[u_{i}, u_{j}]\otimes v_{j}\otimes [v_{i}, x]-u_{i}\otimes [u_{j}, v_{i}]\otimes [v_{j}, x]\big)\\
&&=(\mathrm{id}+\xi+\xi^{2})\Big( -\big(\mathrm{id}\otimes\mathrm{id}\otimes \mathcal{R}_{[\cdot,\cdot]}(x)\big)(\mathrm{id}\otimes\tau)\mathrm{Al}(r)-\sum_{i,j}( v_{i}\otimes [u_{j}, u_{i}]\otimes [v_{j}, x]\\
&&\ \ +v_{i}\otimes [v_{j}, u_{i}]\otimes [u_{j}, x]+u_{i}\otimes [u_{j}, v_{i}]\otimes [v_{j}, x]+u_{i}\otimes [v_{j}, v_{i}]\otimes [u_{j}, x])\Big)\\
&&=-(\mathrm{id}+\xi+\xi^{2})\Bigg(\big(\mathrm{id}\otimes\mathrm{id}\otimes \mathcal{R}_{[\cdot,\cdot]}(x)\big)(\mathrm{id}\otimes\tau)\mathrm{Al}(r)\\
&&\ \ +\sum_{j} \big(\mathrm{id}\otimes\mathcal{R}_{[\cdot,\cdot]}(v_{j})\otimes\mathcal{R}_{[\cdot,\cdot]}(x)\big)\Big(u_{j}\otimes\big(r+\tau(r)\big)\Big)\\
&&~~+\big(\mathrm{id}\otimes\mathcal{R}_{[\cdot,\cdot]}(u_{j})\otimes\mathcal{R}_{[\cdot,\cdot]}(x)\big)\Big(v_{j}\otimes\big(r+\tau(r)\big)\Big)\bigg).
\end{eqnarray*}
In conclusion, we have
\begin{eqnarray*}
&&(\mathrm{id}+\xi+\xi^{2})(\tau\otimes\mathrm{id}-\mathrm{id}^{\otimes 3})(\delta_{r}\otimes\mathrm{id})\delta_{r}(x)\\
&&=-(\mathrm{id}+\xi+\xi^{2})\bigg( \big(\mathrm{id}\otimes\mathrm{id}\otimes\mathcal{R}_{[\cdot,\cdot]}(x)\big) (\mathrm{id}\otimes\tau)\mathrm{Al}(r)+
\sum_{j} h([u_{j}, x]-[x, u_{j}])\big(r+\tau(r)\big)\otimes v_{j}\notag\\
&&\ \ +\tau h(u_{j})\big( r+\tau(r)\big) \otimes v_{j}
+ v_{j}\otimes\big(\mathcal{R}_{[\cdot,\cdot]}(u_{j})\otimes\mathrm{id}\big)h(x)\big(r+\tau(r)\big)\notag\\
&&\ \ +\big(\mathrm{id}\otimes\mathcal{R}_{[\cdot,\cdot]}(v_{j})\otimes\mathcal{R}_{[\cdot,\cdot]}(x)\big)\Big(u_{j}\otimes\big(r+\tau(r)\big)\Big)+
\big(\mathrm{id}\otimes\mathcal{R}_{[\cdot,\cdot]}(u_{j})\otimes\mathcal{R}_{[\cdot,\cdot]}(x)\big)\Big(v_{j}\otimes\big(r+\tau(r)\big)\Big)\Bigg).
\end{eqnarray*}
Hence the conclusion follows.

{\bf (b)} For all $x,y\in A$, we have
    \begin{eqnarray*}
\delta_{r}([x, y])&=&\sum_{i} [u_{i}, [[x, y], v_{i}]]-[[x, y], u_{i}]\otimes v_{i}-u_{i}\otimes [v_{i},[x, y]],\\
-\delta_{r}([y, x])&=&\sum_{i} -[[u_{i}, [y, x]], v_{i}]+[[y, x], u_{i}]\otimes v_{i}+u_{i}\otimes [v_{i},[y, x]],\\
\big( \mathcal{R}_{[\cdot,\cdot]}(x)\otimes\mathrm{id} \big)\delta_{r}(y)&=&\sum_{i} [[u_{i}, y], x]\otimes v_{i}-[[y, u_{i}], x]\otimes v_{i}-[u_{i}, x]\otimes [v_{i}, y],\\
-\big( \mathcal{R}_{[\cdot,\cdot]}(y)\otimes\mathrm{id} \big)\delta_{r}(x)&=&\sum_{i} -[[u_{i}, x], y]\otimes v_{i}+[[x, u_{i}], y]\otimes v_{i}+[u_{i}, y]\otimes [v_{i}, x].
    \end{eqnarray*}
Hence
\begin{eqnarray*}
&&    (\tau-\mathrm{id}^{2})\big( \delta_{r}([x, y]-[y, x])+(\mathcal{R}_{[\cdot,\cdot]}(x)\otimes\mathrm{id})\delta_{r}(y)-(\mathcal{R}_{[\cdot,\cdot]}(y)\otimes\mathrm{id})\delta_{r}(x) \big)\\
&&\overset{\eqref{eq:0-Alia}}{=}\sum_{i} v_{i}\otimes [u_{i},[x, y]]-[v_{i},[x, y]]\otimes u_{i}-v_{i}\otimes [u_{i},[y, x]]+[v_{i},[y, x]]\otimes u_{i}\\
&&\ \ -[v_{i}, y]\otimes [u_{i}, x]+[v_{i}, x]\otimes [u_{i}, y]-[u_{i},[x, y]]\otimes v_{i}+u_{i}\otimes [v_{i},[x, y]]\\
&&\ \ +[u_{i},[y, x]]\otimes v_{i}-u_{i}\otimes [v_{i},[y, x]]+[u_{i}, x]\otimes [v_{i}, y]-[u_{i}, y]\otimes [v_{i}, x].
\end{eqnarray*}
Noticing that
\begin{eqnarray*}
&&\sum_{i} v_{i}\otimes [u_{i},[x, y]]-[v_{i},[x, y]]\otimes u_{i}-[u_{i},[x, y]]\otimes v_{i}+u_{i}\otimes [v_{i},[x, y]]\\
&&=\tau h([x, y])\big( r+\tau(r)\big)+\sum_{i} v_{i}\otimes [[x, y], u_{i}]+u_{i}\otimes[[x, y], v_{i}],\\
&&\sum_{i} -v_{i}\otimes [u_{i},[y, x]]+[v_{i},[y, x]]\otimes u_{i}+[u_{i},[y, x]]\otimes v_{i}-u_{i}\otimes [v_{i},[y, x]]\\
&&=-\Big( \tau h([y, x])\big( r+\tau(r)\big)+\sum_{i} v_{i}\otimes [[y, x], u_{i}]+u_{i}\otimes[[y, x], v_{i}]\Big),\\
&&\sum_{i} v_{i}\otimes[[x, y], u_{i}]-v_{i}\otimes[[y, x], u_{i}]\\
&&\overset{\eqref{eq:0-Alia}}{=}\sum_{i} v_{i}\otimes [[x, u_{i}], y]-v_{i}\otimes [[u_{i}, x], y]+v_{i}\otimes[[u_{i}, y], x]-v_{i}\otimes[[y, u_{i}], x],\\
&& \sum_{i} u_{i}\otimes[[x, y], v_{i}]-u_{i}\otimes[[y, x], v_{i}]\\
&&\overset{\eqref{eq:0-Alia}}{=}\sum_{i} u_{i}\otimes [[x, v_{i}], y]-u_{i}\otimes [[v_{i}, x], y]+u_{i}\otimes[[v_{i}, y], x]-u_{i}\otimes[[y, v_{i}], x], \\
&&\sum_{i} [v_{i}, x]\otimes [u_{i}, y]+[u_{i}, x]\otimes [v_{i}, y]+v_{i}\otimes [[x, u_{i}], y]\\
&&-v_{i}\otimes[[u_{i}, x], y]+u_{i}\otimes[[x, v_{i}], y]-u_{i}\otimes [[v_{i}, x], y]\\
&&=-\big(\mathrm{id}\otimes\mathcal{R}_{[\cdot,\cdot]}(y)\big)\tau h(x)\big( r+\tau(r)\big),\\
&&\sum_{i} -[v_{i}, y]\otimes [u_{i}, x]-[u_{i}, y]\otimes [v_{i}, x]-v_{i}\otimes [[y, u_{i}], x]\\
&&+v_{i}\otimes[[u_{i}, y], x]-u_{i}\otimes[[y, v_{i}], x]+u_{i}\otimes [[v_{i}, y], x]\\
&&= \big(\mathrm{id}\otimes\mathcal{R}_{[\cdot,\cdot]}(x)\big)\tau h(y)\big( r+\tau(r)\big),
\end{eqnarray*}
we finally have
\begin{eqnarray*}
    &&    (\tau-\mathrm{id}^{2})\big( \delta_{r}([x, y]-[y, x])+(\mathcal{R}_{[\cdot,\cdot]}(x)\otimes\mathrm{id})\delta_{r}(y)-(\mathcal{R}_{[\cdot,\cdot]}(y)\otimes\mathrm{id})\delta_{r}(x) \big)\\
    &&=\tau h([x, y]-[y, x])\big(r+\tau(r)\big)-\big(\mathrm{id}\otimes\mathcal{R}_{[\cdot,\cdot]}(y)\big)\tau h(x)\big( r+\tau(r)\big)\\
    &&+\big(\mathrm{id}\otimes\mathcal{R}_{[\cdot,\cdot]}(x)\big)\tau h(y)\big( r+\tau(r)\big).
\end{eqnarray*}
Hence the conclusion follows.
\end{proof}

By Proposition \ref{pro:4.2}, we have

\begin{cor}\label{cor:4.3}
    Let $(A,[\cdot,\cdot])$ be a left-Alia algebra and $r\in A\otimes A$.
    Suppose that $r$ is an antisymmetric solution of the left-Alia Yang-Baxter equation, that is,
    \begin{equation}
   r+\tau(r)=0,\; \mathrm{Al}(r)=0.
    \end{equation}
    Then $(A,[\cdot,\cdot],\delta_{r})$ is a left-Alia bialgebra, where $\delta_{r}:A\rightarrow A\otimes A$ is given by \eqref{eq:co}.
\end{cor}

\begin{defi}
    Let $(A,[\cdot,\cdot])$ be a left-Alia algebra.
    If $r\in A\otimes A$ is an antisymmetric  solution of the left-Alia Yang-Baxter equation,
    then we say the resulting left-Alia bialgebra $(A,[\cdot,\cdot],\delta_{r})$ from Corollary \ref{cor:4.3}  is {\bf  triangular}.
\end{defi}

\section{Relative Rota-Baxter operators of left-Alia algebras}
In this section, we study relative Rota-Baxter operators of left-Alia algebras, which serve as operator forms of the left-Alia Yang-Baxter equation.
\begin{defi}
Let $(A,[\cdot,\cdot])$ be a left-Alia algebra with a representation $(l,r,V)$.
A linear map $T:V\rightarrow A$ is called a {\bf relative Rota-Baxter operator} of $(A,[\cdot,\cdot])$ associated to $(l,r,V)$ if the following equation holds:
\begin{equation}\label{eq:oop}
    [T(u),T(v)]=T\big(l(Tu)v+r(Tv)u\big),\;\forall u,v\in V.
\end{equation}
\end{defi}

For vector spaces $V$ and $ A$, we  can identify an element $r\in V\otimes A$ as a linear map $r^{\sharp}:V^{*}\rightarrow A$ by
\begin{equation}
    \langle r^{\sharp}(u^{*}), a^{*}\rangle=\langle r, u^{*}\otimes a^{*}\rangle, \;\forall u^{*}\in V^{*}, a^{*}\in A^{*}.
\end{equation}
Now we study the operator form of $r$ in a triangular left-Alia bialgebra $(A,[\cdot,\cdot],\delta_{r})$.

\begin{pro} \label{pro:4.4}
Let $(A,[\cdot,\cdot]_A )$ be a left-Alia algebra and $r=\sum\limits_{i} u_{i}\otimes v_{i}\in A\otimes A$ be antisymmetric.
Let $\delta_{r}:A\rightarrow A\otimes A$ be a co-multiplication given by \eqref{eq:co} and $[\cdot,\cdot]_{A^*} :A^{*}\otimes A^{*}\rightarrow A^{*}$ be the linear dual of $\delta_{r}$.
Then we have
\begin{equation}\label{eq:com}
    [a^{*},  b^{*}]_{A^*}=\mathcal{L}^{*}_{[\cdot,\cdot]_A }\big(r^{\sharp}(a^{*})\big)b^{*}
    +(\mathcal{L}^{*}_{[\cdot,\cdot]_A }-\mathcal{R}^{*}_{[\cdot,\cdot]_A })\big( r^{\sharp}(b^{*}) \big)a^{*},\;\forall a^{*},b^{*}\in A^{*}.
\end{equation}
Moreover, we have
{\small
\begin{equation}\label{eq:homo1}
  \langle r^{\sharp}([a^{*}, b^{*}]_{A^*})-[r^{\sharp}(a^{*}),  r^{\sharp}(b^{*})]_{A}, c^{*}\rangle=\langle a^{*}\otimes b^{*}\otimes c^{*}, (\tau\otimes\mathrm{id})\mathrm{Al}(r) \rangle.
\end{equation}
}
\end{pro}
\begin{proof}
Let $a^{*}, b^{*},c^{*}\in A^{*}, x\in A$.
Then we have
\begin{eqnarray*}
&&\langle [a^{*},  b^{*}]_{A^*},x\rangle=\langle a^{*}\otimes b^{*}, \delta_{r}(x)\rangle\\
&&=\langle a^{*}\otimes b^{*}, \big( (\mathcal{R}_{[\cdot,\cdot]_A }-\mathcal{L}_{[\cdot,\cdot]_A })(x)\otimes\mathrm{id}-\mathrm{id}\otimes\mathcal{R}_{[\cdot,\cdot]_A }(x)\big) r\rangle\\
&&=\langle r,  (\mathcal{L}^{*}_{[\cdot,\cdot]_A }-\mathcal{R}^{*}_{[\cdot,\cdot]_A})(x) a^{*}\otimes b^{*}+a^{*}\otimes\mathcal{R}^{*}_{[\cdot,\cdot]_A }(x)    b^{*}\rangle\\
&&=\langle r^{\sharp}(a^{*}),\mathcal{R}^{*}_{[\cdot,\cdot]_A }(x)b^{*}\rangle-\langle r^{\sharp}(b^{*}), (\mathcal{L}^{*}_{[\cdot,\cdot]_A }-\mathcal{R}^{*}_{[\cdot,\cdot]_A })(x)a^{*}\rangle\\
&&=-\langle [r^{\sharp}(a^{*}), x]_{A}, b^{*}\rangle+\langle [x, r^{\sharp}(b^{*})]_{A}-[r^{\sharp}(b^{*}), x]_A,a^{*}\rangle\\
&&=\langle x, \mathcal{L}^{*}_{[\cdot,\cdot]_A }\big(r^{\sharp}(a^{*})\big)b^{*}
    +(\mathcal{L}^{*}_{[\cdot,\cdot]_A }-\mathcal{R}^{*}_{[\cdot,\cdot]_A })\big( r^{\sharp}(b^{*}) \big)a^{*}\rangle.
\end{eqnarray*}
Hence \eqref{eq:com} holds. Moreover, we have
\begin{eqnarray*}
&&\langle r^{\sharp}([a^{*}, b^{*}]_{A^*}),c^{*}\rangle=\sum_{i}\langle u_{i}, [a^{*}, b^{*}]_{A^*}\rangle\langle v_{i}, c^{*}\rangle=\sum_{i}\langle\delta(u_{i}),a^{*}\otimes b^{*}\rangle\langle v_{i}, c^{*}\rangle\\
&&=\sum_{i,j}\langle  [u_{j}, u_{i}]_{A} \otimes v_{j}-[u_{i}, u_{j}]_A\otimes v_{j}-u_{j}\otimes [v_{j}, u_{i}]_A, a^{*}\otimes b^{*}\rangle\langle v_{i},c^{*}\rangle\\
&&=\sum_{i,j}\langle a^{*}\otimes b^{*}\otimes c^{*}, [u_{j}, u_{i}]_A\otimes v_{j}\otimes v_{i}-[u_{i},  u_{j}]_A\otimes v_{j}\otimes v_{i}-u_{j}\otimes [v_{j}, u_{i}]_A\otimes v_{i}\rangle.
\end{eqnarray*}
Similarly we have
\begin{eqnarray*}
    \langle r^{\sharp}([a^{*}), r^{\sharp}(b^{*})]_A, c^{*}\rangle=\sum_{i,j}\langle a^{*}\otimes b^{*}\otimes c^{*}, u_{i}\otimes u_{j}\otimes [v_{i}, v_{j}]_A\rangle.
\end{eqnarray*}
Thus
{\small
\begin{eqnarray*}
&&\langle r^{\sharp}([a^{*}, b^{*}]_{A^*})- [r^{\sharp}(a^{*}), r^{\sharp}(b^{*})]_A, c^{*}\rangle\\
&&=\sum_{i,j}\langle a^{*}\otimes b^{*}\otimes c^{*}, ([u_{j}, u_{i}]_A-[u_{i},  u_{j}]_A)\otimes v_{j}\otimes v_{i} -u_{j}\otimes [v_{j}, u_{i}]_A\otimes v_{i}-u_{i}\otimes u_{j}\otimes [v_{i}, v_{j}]_A\rangle\\
&&=\langle a^{*}\otimes b^{*}\otimes c^{*}, (\tau\otimes\mathrm{id})\mathrm{Al}(r)+\sum_{j}h(u_{j})\big( r+\tau(r)\big)\otimes v_{j}\rangle\\
&&=\langle a^{*}\otimes b^{*}\otimes c^{*}, (\tau\otimes\mathrm{id})\mathrm{Al}(r) \rangle.
\end{eqnarray*}}
that is, \eqref{eq:homo1} holds.
\end{proof}

\begin{cor}
    Let $(A,[\cdot,\cdot] )$ be a left-Alia algebra and $r=\sum\limits_{i} u_{i}\otimes v_{i}\in A\otimes A$ be antisymmetric.
Let $\delta_{r}:A\rightarrow A\otimes A$ be a co-multiplication given by \eqref{eq:co}.
Then the following statements are equivalent:
\begin{enumerate}
    \item $r$ is a solution of the left-Alia Yang-Baxter equation in $(A,[\cdot,\cdot])$ such that $(A,[\cdot,\cdot],\delta_{r})$ is a triangular left-Alia bialgebra.
    \item $r^{\sharp}:A^{*}\rightarrow A$ is a relative Rota-Baxter operator of $(A,[\cdot,\cdot])$ associated to $(\mathcal{L}^{*}_{[\cdot,\cdot]},\mathcal{L}^{*}_{[\cdot,\cdot]}-\mathcal{R}^{*}_{[\cdot,\cdot]},A^{*})$, that is,
\begin{equation}
[r^{\sharp}(a^{*}),  r^{\sharp}(b^{*})]=r^{\sharp} \Big(\mathcal{L}^{*}_{[\cdot,\cdot] }\big(r^{\sharp}(a^{*})\big)b^{*}
    +(\mathcal{L}^{*}_{[\cdot,\cdot] }-\mathcal{R}^{*}_{[\cdot,\cdot] })\big( r^{\sharp}(b^{*}) \big)a^{*}\Big),\;\forall a^{*},b^{*}\in A^{*}.
\end{equation}
\end{enumerate}
In this case,  $r^{\sharp}:(A^{*},[\cdot,\cdot]_{A^*})\rightarrow (A,[\cdot,\cdot])$ is a homomorphism
 of left-Alia algebras, where $[\cdot,\cdot]_{A^*} :A^{*}\otimes A^{*}\rightarrow A^{*}$ is the linear dual of $\delta_{r}$.
\end{cor}
\begin{proof}
    It follows from Proposition \ref{pro:4.4}.
\end{proof}

 Now we consider the case when $r^{\sharp}$ is moreover an isomorphism.

\delete{
Suppose $r\in A\otimes A$ is antisymmetric and $r^{\sharp}:A^{*}\rightarrow A$ is a linear isomorphism.
Let $\omega$ be the nondegenerate antisymmetric bilinear form on $A$ given by
\begin{equation}
    \omega(x,y)=\langle (r^{\sharp})^{-1}x,y\rangle,\;\forall x,y\in A.
\end{equation}
Then $r$ is a solution of the left-Alia Yang-Baxter equation in $(A,[\cdot,\cdot])$ such that $(A,\circ,\delta_{r})$ is a triangular left-Alia bialgebra if and only if}

\begin{pro}\label{pro:4.8}
Let $(A,[\cdot,\cdot])$ be a left-Alia algebra.
Then there is an antisymmetric $r\in A\otimes A$ such that $r^{\sharp}:A^{*}\rightarrow A$ is an   invertible relative Rota-Baxter operator associated to $(\mathcal{L}^{*}_{[\cdot,\cdot]},\mathcal{L}^{*}_{[\cdot,\cdot]}-\mathcal{R}^{*}_{[\cdot,\cdot]},A^{*})$ if and only if there is a nondegenerate antisymmetric bilinear form $\omega$ on $A$ such that
the following equation holds:
\begin{equation}\label{eq:2-coc}
\omega(x, [y, z]-[z, y])=\omega([x, z],y)-\omega([x, y],z),\;\forall x,y,z\in A.
\end{equation}
\end{pro}
\begin{proof}
Suppose $r\in A\otimes A$ is antisymmetric and $r^{\sharp}$ is an invertible relative Rota-Baxter operator of $(A,[\cdot,\cdot])$ associated to $(\mathcal{L}^{*}_{[\cdot,\cdot]},\mathcal{L}^{*}_{[\cdot,\cdot]}-\mathcal{R}^{*}_{[\cdot,\cdot]},A^{*})$.
Define a bilinear form $\omega$ on $A$ by
\begin{equation}\label{eq:34}
    \omega(x,y)=\langle (r^{\sharp})^{-1}x,y\rangle,\;\forall x,y\in A.
\end{equation}
Then $\omega$ is antisymmetric and nondegenerate.
Let $a^{*},b^{*},c^{*}\in A^{*}$ and $x=r^{\sharp}(a^{*}), y=r^{\sharp}(b^{*}), z=r^{\sharp}(c^{*})$.
Then we have
\begin{eqnarray*}
&&\langle [r^{\sharp}(a^{*}), r^{\sharp}(b^{*})], c^{*}\rangle=\langle [r^{\sharp}(a^{*}),  r^{\sharp}(b^{*})], (r^{\sharp})^{-1}z\rangle=\omega(z,[x, y]),\\
&&\langle r^{\sharp}\big( \mathcal{L}^{*}_{[\cdot,\cdot] }\big(r^{\sharp}(a^{*})\big)b^{*} \big),c^{*}\rangle=\langle r^{\sharp}\big(\mathcal{L}^{*}_{\circ}(x)b^{*}\big),  (r^{\sharp})^{-1}z\rangle=-\langle \mathcal{L}^{*}_{[\cdot,\cdot]}(x)b^{*},z\rangle\\
&&=\langle b^{*},[x, z]\rangle=\langle (r^{\sharp})^{-1}y,[x, z]\rangle =\omega(y,[x, z]),\\
&&\langle r^{\sharp}\Big( (\mathcal{L}^{*}_{[\cdot,\cdot] }-\mathcal{R}^{*}_{[\cdot,\cdot] })\big(r^{\sharp}(b^{*})\big)a^{*} \Big),c^{*}\rangle=
\langle r^{\sharp}\big( (\mathcal{L}^{*}_{[\cdot,\cdot] }-\mathcal{R}^{*}_{[\cdot,\cdot] })(y)a^{*} \big),  (r^{\sharp})^{-1}z\rangle\\
&&=-\langle (\mathcal{L}^{*}_{[\cdot,\cdot] }-\mathcal{R}^{*}_{[\cdot,\cdot] })(y)a^{*},z\rangle\\
&&=\langle a^{*}, [y, z]-[z, y]\rangle=\langle (r^{\sharp})^{-1}x, [y, z]-[z, y]\rangle=\omega(x,[y, z]-[z, y]).
\end{eqnarray*}
Hence by \eqref{eq:homo1} in Proposition \ref{pro:4.4}, we have \eqref{eq:2-coc}.

Conversely, suppose that $\omega$ is a nondegenerate antisymmetric bilinear form on $(A,[\cdot,\cdot])$ satisfying \eqref{eq:2-coc}.
Then we can routinely check that $r^{\sharp}:A^{*}\rightarrow A$ given by \eqref{eq:34} is  an invertible relative Rota-Baxter operator of $(A,[\cdot,\cdot])$ associated to $(\mathcal{L}^{*}_{[\cdot,\cdot]},\mathcal{L}^{*}_{[\cdot,\cdot]}-\mathcal{R}^{*}_{[\cdot,\cdot]},A^{*})$.
\end{proof}

Recall that a {\bf Connes cocycle} \cite{Bai2010} on an associative algebra $(A,\cdot)$ is an antisymmetric bilinear form $\omega$ such that
\begin{equation}\label{eq:cc}
    \omega(x\cdot y,z)+\omega(y \cdot z, x)+\omega(z\cdot x,y)=0,\;\forall x,y,z\in A.
\end{equation}

\begin{pro}
Suppose that  $\omega$ is a nondegenerate Connes cocycle on a commutative associative algebra $(A,\cdot)$.
Let $f:A\rightarrow A$ be a linear map and $\hat{f}:A\rightarrow A$ be the adjoint map of $f$ with respect to $\omega$, that is,
\begin{equation*}
    \omega\big(\hat{f}(x),y\big)=\omega\big(x,f(y)\big),\;\forall x,y\in A.
\end{equation*}
Then $\omega$ satisfies \eqref{eq:2-coc}, where  $(A,[\cdot,\cdot])$ is the special left-Alia algebra given by \eqref{eq:hat}.
\end{pro}
\begin{proof}
For all $x,y,z\in A$, we have
{\small
\begin{eqnarray*}
&&\omega(x,[y, z]-[z, y])-\omega([x, z],y)+\omega([x, y],z)\\
&&=\omega\big( x,y\cdot f(z)-z\cdot f(y)\big)-\omega\big(x\cdot f(z)-\hat{f}(x\cdot z),y\big)+\omega\big(x\cdot f(y)-\hat{f}(x\cdot y),z\big)\\
&&=\omega\big(x,y\cdot f(z)-z\cdot f(y)\big)-\omega\big(x\cdot f(z),y\big)+\omega\big(x\cdot z,f(y)\big)+\omega\big(x\cdot f(y),z\big)-\omega(x\cdot y,f(z)\big)\\
&&\overset{\eqref{eq:cc}}{=}0.
\end{eqnarray*}}
\end{proof}

\begin{thm}\label{thm:4.10}
Let $(A,[\cdot,\cdot])$ be a left-Alia algebra with a representation $(l,r,V)$.
Let $T:V\rightarrow A$ be a linear map and $T_{\sharp}\in V^{*}\otimes A\subset (A\oplus V^{*})\otimes (A\oplus V^{*})$ be given by
\begin{equation}
    \langle T_{\sharp}, u\otimes a^{*}\rangle=\langle T(u),a^{*}\rangle,\;\forall u\in V, a^{*}\in A^{*}.
\end{equation}
Then $r=T_{\sharp}-\tau(T_{\sharp})$ is an antisymmetric solution of the left-Alia Yang-Baxter equation in $A\ltimes_{l^{*}, l^{*}-r^{*}}V^{*}$ if and only if $T$ is a relative Rota-Baxter operator of $(A,[\cdot,\cdot])$ associated to $(l,r,V)$.
\end{thm}
\begin{proof}
Let $\{ v_{1},\cdots,v_{n} \}$ be a basis of $V$ and $\{ v^{*}_{1},\cdots,v^{*}_{n} \}$ be the dual basis.
Then we have
\begin{eqnarray*}
    &&T_{\sharp}=\sum_{i=1}^{n}v^{*}_{i}\otimes T v_{i} \in V^{*}\otimes A\subset (A\oplus V^{*})\otimes (A\oplus V^{*}),\\
    &&r=T_{\sharp}-\tau(T_{\sharp})=\sum_{i=1}^{n}v^{*}_{i}\otimes T v_{i} -T v_{i} \otimes v^{*}_{i}.
\end{eqnarray*}
Note that
\begin{equation}
l^{*}(Tv_{i})v^{*}_{j}=-\sum_{k=1}^{n}\langle v^{*}_{j}, l(Tv_{i})v_{k}\rangle v^{*}_{k},\;r^{*}(Tv_{i})v^{*}_{j}=-\sum_{k=1}^{n}\langle v^{*}_{j}, r(Tv_{i})v_{k}\rangle v^{*}_{k}.
\end{equation}
Then we have
{\small
\begin{eqnarray*}
 \mathrm{Al}(r)&=&\sum_{i,j=1}^{n} [Tv_{i}, Tv_{j}]\otimes v^{*}_{i}\otimes v^{*}_{j}-[v^{*}_{i},Tv_{j}]_d\otimes Tv_{i}\otimes v^{*}_{j}-[Tv_{i},v^{*}_{j}]_d\otimes v^{*}_{i}\otimes Tv_{j}\\
&&+v^{*}_{i}\otimes\big( [v^{*}_{j}, Tv_{i}]_d-[Tv_{i},v^{*}_{j}]_d\big)\otimes Tv_{j}
+v^{*}_{i}\otimes ([Tv_{i}, Tv_{j}]-[Tv_{j}, Tv_{i}])\otimes v^{*}_{j}\\
&&
+Tv_{i}\otimes \big( [Tv_{j},v^{*}_{i}]_d-[v^{*}_{i},Tv_{j}]_d \big)\otimes v^{*}_{j}\\
&&-v^{*}_{i}\otimes v^{*}_{j}\otimes [Tv_{j}, Tv_{i}]+v^{*}_{i}\otimes Tv_{j}\otimes [v^{*}_{j},Tv_{i}]_d+Tv_{i}\otimes v^{*}_{j}\otimes [Tv_{j},v^{*}_{i}]_d\\
&=&\sum_{i,j} [Tv_{i}, Tv_{j}]\otimes v^{*}_{i}\otimes v^{*}_{j}-(l^{*}-r^{*})(Tv_{j})v^{*}_{i}\otimes Tv_{i}\otimes v^{*}_{j}-l^{*}(Tv_{i})v^{*}_{j}\otimes v^{*}_{i}\otimes Tv_{j}\\
&&-v^{*}_{i}\otimes r^{*}(Tv_{i})v^{*}_{j}\otimes Tv_{j}+v^{*}_{i}\otimes ( [Tv_{i}, Tv_{j}]-[Tv_{j}, Tv_{i}]  )\otimes v^{*}_{j}+Tv_{i}\otimes r^{*}(Tv_{j})v^{*}_{i}\otimes v^{*}_{j}\\
&&-v^{*}_{i}\otimes v^{*}_{j}\otimes [Tv_{j}, Tv_{i}]+v^{*}_{i}\otimes Tv_{j}\otimes (l^{*}-r^{*}) (Tv_{i})v^{*}_{j}+Tv_{i}\otimes v^{*}_{j}\otimes l^{*}(Tv_{j})v^{*}_{i}\\
&=&\sum_{i,j} \Big( [Tv_{i}, Tv_{j}]-T\big(l(Tv_{i})v_{j}+r(Tv_{j})v_{i}\big)  \Big)\otimes v^{*}_{i}\otimes v^{*}_{j}\\
&&+v^{*}_{i}\otimes ([Tv_{i}, Tv_{j}]-[Tv_{j}, Tv_{i}]+T\big((l-r)(Tv_{j})v_{i}+(r-l)(Tv_{i})v_{j}\big)\Big)\otimes v^{*}_{j}\\
&&+v^{*}_{i}\otimes v^{*}_{j}\otimes \Big ( T\big(l(Tv_{j})v_{i}+r(Tv_{i})v_{j}\big)-[Tv_{j}, Tv_{i} ] \Big).
\end{eqnarray*}}
Hence the conclusion follows.
\end{proof}

\section{Pre-left-Alia algebras}

In this section,  we first introduce the notion of a pre-left-Alia algebra,  which is the underlying algebraic structure of
a relative Rota-Baxter operator with respect to a representation $(l,r,V)$ of a left-Alia algebra. Then we prove that there is a
natural antisymmetric solution of the left-Alia Yang-Baxter equation in the left-Alia algebra $A\ltimes_{\mathcal{L}^{*}_{\succ},\mathcal{L}^{*}_{\succ}-\mathcal{R}^{*}_{\prec} }A^{*}$.
\begin{defi}
A {\bf pre-left-Alia algebra} is a triple $(A,\succ,\prec)$, where $A$
 is a vector space and $\succ,\prec:A\otimes A\rightarrow A$ are multiplications such that the following equation holds:
 {\small
 \begin{equation}\label{eq:pre}
    (x\succ y+x\prec y)\succ z-(y\succ x+y\prec x)\succ z+(y\succ z-z\prec y)\prec x +(z\prec x-x\succ z)\prec y =0,
    \end{equation}}
 for all $x,y,z\in A$.
 \end{defi}

Recall a \textbf{Zinbiel algebra} \cite{Lod}  is a  vector space $A$ together with a multiplication $\star:A\otimes A\rightarrow A$ such that the following equation holds:
    \begin{equation}\label{eq:Zinb}
        x\star(y\star z)=(y\star x)\star z+(x\star y)\star z, \;\;\forall x,y,z\in A.
    \end{equation}
From the operadic viewpoint, Zinbiel algebras are known as the successors of commutative associative algebras, as well as the Koszul dual to Leibniz algebras (see \cite{Lod2012}). We shall show that there are pre-left-Alia algebras constructed Zinbiel algebras with linear maps.

\begin{pro}
Let $(A,\star)$ be a Zinbiel algebra and $f,g:A\rightarrow A$ be linear maps.
Then there is a pre-left-Alia algebra $(A\succ,\prec)$ given by
\begin{equation*}
    x\succ y=x\star f(y)+g(x\star y),\; x\prec y=f(y)\star x+g(y\star x),\;\forall x,y\in A.
\end{equation*}
\end{pro}
\begin{proof}
Let $x,y,z\in A$.
We first observe that
\begin{equation*}
    x\succ y-y\prec x=x\star f(y)-f(x)\star y.
\end{equation*}
Then we have
{\small
\begin{eqnarray*}
&&(x\succ y+x\prec y)\succ z-(y\succ x+y\prec x)\succ z+(y\succ z-z\prec y)\prec x +(z\prec x-x\succ z)\prec y\\
&&=\big(x\star f(y)-f(x)\star y\big)\succ z-\big(y\star f(x)-f(y)\star x\big)\succ z\\
&&\ \ +\big( y\star f(z)-f(y)\star z\big)\prec x+\big(f(x)\star z-z\star f(x)\big)y\\
&&=\big(x\star f(y)-f(x)\star y\big)\star z+g\Big( \big(x\star f(y)\big)\star z-\big(f(x)\star y\big)\star z  \Big)\\
&&\ \ -\big(y\star f(x)-f(y)\star x\big)\star z-g\Big( \big(y\star f(x)\big)\star z-\big(f(y)\star x\big)\star z  \Big)\\
&&\ \ +f(x)\star\big( y\star f(z)-f(y)\star z\big)+g\Big(x\star \big(y\star f(z)\big)-x\star \big(f(y)\star z\big)\Big)\\
&&\ \ +f(y)\star\big( f(x)\star z-f(z)\star x\big)+g\Big(y\star \big(f(x)\star z\big)-y\star \big(x\star f(z)\big)\Big)\\
&&\overset{\eqref{eq:Zinb}}{=}0.
\end{eqnarray*}
}
Hence $(A\succ,\prec)$ is a pre-left-Alia algebra.
\end{proof}

On the other hand, recall a {\bf pre-Lie algebra}  \cite{Bur, Bai2021.2}  is a  vector space $A$ together with a multiplication $\star:A\otimes A\rightarrow A$ such that the following equation holds:
    \begin{equation}\label{eq:pre-Lie}
     (x\star y)\star z-x\star(y\star z)=(y\star x)\star z-y\star (x\star z),\;\forall x,y,z\in A.
    \end{equation}
Pre-Lie algebras play an important role in diverse areas of mathematics and mathematical physics.
From the operadic viewpoint, pre-algebras are known as the successors of Lie algebras.

 \begin{pro}\label{pro:6.3}
    Let $(A,\star)$ be a pre-Lie algebra.
    Then there is a pre-left-Alia algebra $(A\succ,\prec)$ given by
    \begin{equation*}
        x\succ y=x\star y,\; x\prec y=-y\star x,\;\forall x,y\in A.
    \end{equation*}
 \end{pro}
\begin{proof}
For all $x,y,z\in A$, we have
{\small
\begin{eqnarray*}
&&(x\succ y+x\prec y)\succ z-(y\succ x+y\prec x)\succ z+(y\succ z-z\prec y)\prec x +(z\prec x-x\succ z)\prec y\\
&&=2(x\star y-y\star x)\star z- 2x\star(y\star z)+2y\star (x\star z)\\
&&\overset{\eqref{eq:pre-Lie}}{=}0.
\end{eqnarray*}}
Hence $(A\succ,\prec)$ is a pre-left-Alia algebra.
\end{proof}

\begin{pro}\label{pro:4.14}
Let $(A,\succ,\prec)$ be a pre-left-Alia algebra.
Then the multiplication $[\cdot,\cdot]:A\otimes A\rightarrow A$ given by
\begin{equation}
[x, y]=x\succ y+x\prec y,\;\forall x,y\in A,
\end{equation}
defines a left-Alia algebra which is called the {\bf sub-adjacent left-Alia algebra} of $(A,\succ,\prec)$, and $(A,\succ,\prec)$ is called a {\bf compatible pre-left-Alia algebra} of $(A,[\cdot,\cdot])$.
Moreover, $(\mathcal{L}_{\succ},\mathcal{R}_{\prec},A)$ is a representation of $(A,[\cdot,\cdot])$, and the identity map $\mathrm{id}:A\rightarrow A$ is a relative Rota-Baxter operator of $(A,[\cdot,\cdot])$ with respect to  $(\mathcal{L}_{\succ},\mathcal{R}_{\prec},A)$.
\end{pro}
\begin{proof}
For all $x,y,z\in A$, we have
{\small
\begin{eqnarray*}
&&[[x, y], z]+[[y\circ z], x]+[[z,x], y]-[[y, x], z]-[[z,y], x]-[[x,z], y]\\
&&=(x\succ y+x\prec y)\succ z+(x\succ y+x\prec y)\prec z+(y\succ z+y\prec z)\succ x+(y\succ z+y\prec z)\prec x\\
&&\ \ +(z\succ x+z\prec x)\succ y+(z\succ x+z\prec x)\prec y-(y\succ x+y\prec x)\succ z-(y\succ x+y\prec x)\prec z\\
&&\ \ -(z\succ y+z\prec y)\succ x-(z\succ y+z\prec y)\prec x-(x\succ z+x\prec z)\succ y-(x\succ z+x\prec z)\prec y\\
&&\overset{\eqref{eq:pre}}{=}0.
\end{eqnarray*}}
Hence $(A,[\cdot,\cdot])$ is a left-Alia algebra.
The other results are straightforward.
\end{proof}

\begin{lem}\label{lem:4.15}
    Let $T:V\rightarrow A$ be a relative Rota-Baxter operator of a left-Alia algebra with respect to a representation $(l,r,V)$.
    Then there is a pre-left-Alia algebra $(V,\triangleright,\triangleleft)$ given by
    \begin{equation}\label{eq:V-pre}
        u\triangleright v=l(Tu)v,\; u\triangleleft v=r(Tv)u,\;\forall u,v\in V.
    \end{equation}
\end{lem}
\begin{proof}
    For all $u,v,w\in V$, we have
    {\small
    \begin{eqnarray*}
&& (u\triangleright v+u\triangleleft v)\triangleright w-(v\triangleright u+v\triangleleft u)\triangleright w+(v\triangleright w-w\triangleleft v)\triangleleft u+(w\triangleleft u-u\triangleleft w)\triangleleft v\\
&&=lT\big( l(Tu)v+r(Tv)u \big)w-lT\big( l(Tv)u+r(Tu)v\big)w\\
&&\ \ +r(Tu)r(Tv)w-r(Tu)r(Tv)w+r(Tv)r(Tu)w-r(Tv)l(Tu)w\\
&&\overset{\eqref{eq:oop}}{=}l([Tu, Tv])-l([Tv, Tu])w+r(Tu)l(Tv)w-r(Tu)r(Tv)w+r(Tv)r(Tu)w-r(Tv)l(Tu)w\\
&&\overset{\eqref{eq:rep 0-Alia}}{=}0.
    \end{eqnarray*}}
    Hence the conclusion follows.
\end{proof}

\begin{thm}\label{thm:4.16}
Let $(A,[\cdot,\cdot])$ be a left-Alia algebra.
Then there is a compatible pre-left-Alia algebra $(A,\succ,\prec)$ of $(A,[\cdot,\cdot])$ if and only if there exists an invertible relative Rota-Baxter operator $T:V\rightarrow A$ with respect to a representation $(l,r,V)$.
Furthermore, the compatible pre-left-Alia algebra is given by
\begin{equation}\label{eq:42}
x\succ y=T\big(l(x)T^{-1}y\big),\; x\prec y=T\big(r(y)T^{-1}x\big),\;\forall x,y\in A.
\end{equation}
\end{thm}
\begin{proof}
    Let $T:V\rightarrow A$ be an invertible relative Rota-Baxter operator of $(A,[\cdot,\cdot])$ with respect to $(l,r,V)$.
    Then by Lemma \ref{lem:4.15}, there is a pre-left-Alia algebra algebra on $V$ given by \eqref{eq:V-pre}.
    Since $T$ is invertible, there is a pre-left-Alia algebra structure  $(A,\succ,\prec)$  on $A$ given by
    \begin{eqnarray*}
    &&x\succ y=T(T^{-1}x\triangleright T^{-1}y)\overset{\eqref{eq:42}}{=}T\big(l(x)T^{-1}y\big),\\
    &&x\prec y=T(T^{-1}x\triangleleft T^{-1}y)\overset{\eqref{eq:42}}{=}T\big(r(y)T^{-1}x\big).
    \end{eqnarray*}
    Moreover, we have
    \begin{eqnarray*}
    x\succ y+x\prec y=T\big(l(x)T^{-1}y+r(y)T^{-1}x\big)\overset{\eqref{eq:oop}}{=}[x, y],
    \end{eqnarray*}
    that is, $(A,\succ,\prec)$ is a compatible pre-left-Alia algebra of $(A,[\cdot,\cdot])$.

    Conversely, if $(A,\succ,\prec)$ is a compatible pre-left-Alia algebra of $(A,[\cdot,\cdot])$, then the identity map $\mathrm{id}:A\rightarrow A$ is an invertible relative Rota-Baxter operator of the sub-adjacent left-Alia algebra $(A,[\cdot,\cdot])$ with respect to  $(\mathcal{L}_{\succ},\mathcal{R}_{\prec},A)$.
\end{proof}

\begin{cor}
Let $\omega$ be a nondegenerate antisymmetric bilinear form on a left-Alia algebra $(A,[\cdot,\cdot])$ satisfying \eqref{eq:2-coc}.
Then there is a compatible pre-left-Alia algebra  $(A,\succ,\prec)$  given by
\begin{equation}\label{eq:44}
    \omega(x\succ y,z)=-\omega(y,[x,z]),\; \omega(x\prec y,z)=\omega(x,[z,y]-[y,z]),\;\forall x,y,z\in A.
\end{equation}
\end{cor}
\begin{proof}
By Proposition \ref{pro:4.8}, $r^{\sharp}:A^{*}\rightarrow A$ given by \eqref{eq:34} is an invertible relative Rota-Baxter operator of $(A,[\cdot,\cdot])$ with respect to $(\mathcal{L}^{*}_{[\cdot,\cdot]},\mathcal{L}^{*}_{[\cdot,\cdot]}-\mathcal{R}^{*}_{[\cdot,\cdot]},A^{*})$. Then by Theorem \ref{thm:4.16}, there is a compatible pre-left-Alia algebra $(A,\succ,\prec)$ of $(A,[\cdot,\cdot])$ given by
\begin{equation}
    x\succ y=r^{\sharp}\big( \mathcal{L}^{*}_{[\cdot,\cdot]}(x)(r^{\sharp})^{-1}y\big),\;
    x\prec y=r^{\sharp}\big( (\mathcal{L}^{*}_{[\cdot,\cdot]}-\mathcal{R}^{*}_{[\cdot,\cdot]})(y)(r^{\sharp})^{-1}x\big).
\end{equation}
Moreover, we have
{\small
\begin{eqnarray*}
 \omega(x\succ y,z)&=&\omega\Big(  r^{\sharp}\big( \mathcal{L}^{*}_{[\cdot,\cdot]}(x)(r^{\sharp})^{-1}y\big),z \Big)=\langle   \mathcal{L}^{*}_{[\cdot,\cdot]}(x)(r^{\sharp})^{-1}y ,z \rangle=-\langle (r^{\sharp})^{-1}y,[x, z]\rangle= -\omega(y,[x, z]),\\
 \omega(x\prec y,z)&=&\omega\Big(  r^{\sharp}\big( (\mathcal{L}^{*}_{[\cdot,\cdot]}-\mathcal{R}^{*}_{[\cdot,\cdot]})(y)(r^{\sharp})^{-1}x\big),z \Big)=
\langle (\mathcal{L}^{*}_{[\cdot,\cdot]}-\mathcal{R}^{*}_{[\cdot,\cdot]})(y)(r^{\sharp})^{-1}x,z\rangle\\
&=&\langle (r^{\sharp})^{-1}x, [z, y]-[y, z]\rangle=\omega(x, [z, y]-[y, z]).
\end{eqnarray*}}
Hence $(A,\succ,\prec)$ satisfies \eqref{eq:44}.
\end{proof}

\begin{thm}\label{thm:6.8}
Let $(A,\succ,\prec)$ be a pre-left-Alia algebra.
Then
\begin{equation*}
    r=\sum_{i=1}^{n}e^{*}_{i}\otimes e_{i}-e_{i}\otimes e^{*}_{i}
\end{equation*}
is an antisymmetric solution of the left-Alia Yang-Baxter equation in the left-Alia algebra $A\ltimes_{\mathcal{L}^{*}_{\succ},\mathcal{L}^{*}_{\succ}-\mathcal{R}^{*}_{\prec} }A^{*}$,
where $\{e_{1},\cdots,e_{n}\}$ is a basis of $A$ and $\{e^{*}_{1},\cdots,e^{*}_{n}\}$ is  the dual basis.
\end{thm}
\begin{proof}
    Since $(A,\succ,\prec)$ is a pre-left-Alia algebra, the identity map $\mathrm{id}:A\rightarrow A$ is a relative Rota-Baxter operator of the sub-adjacent left-Alia algebra with respect to $(\mathcal{L}_{\succ},\mathcal{R}_{\prec}, A)$.
    Then by Theorem \ref{thm:4.10},
    \begin{equation*}
        r=(\mathrm{id})_{\sharp}-\tau(\mathrm{id}_{\sharp})=\sum_{i=1}^{n}e^{*}_{i}\otimes e_{i}-e_{i}\otimes e^{*}_{i}
    \end{equation*}
    is an antisymmetric solution of the left-Alia Yang-Baxter equation in the left-Alia algebra $A\ltimes_{\mathcal{L}^{*}_{\succ},\mathcal{L}^{*}_{\succ}-\mathcal{R}^{*}_{\prec} }A^{*}$.
\end{proof}

\begin{ex}
Let $(A,\succ,\prec)$ be a 2-dimensional pre-left-Alia algebra with a basis $\{e_1,e_2\}$ and the non-zero multiplications $\succ,\prec$ given by
    \begin{align*}
        &e_1\succ e_2=2e_1,\;e_1\prec e_2=-e_1,\;e_2\succ e_1=e_1,\;e_2\prec e_1=-2e_1, \\
        &e_2\succ e_2=e_1+2e_2,\;e_2\prec e_2=-e_1-2e_2.
    \end{align*}
Then the non-zero multiplications of the sub-adjacent left-Alia algebra of $(A,\succ,\prec)$ are
\begin{equation*}
  [e_1,e_2]=e_1,\; [e_2,e_1]=-e_1.
\end{equation*}
If $\{e_1^*,e_2^*\}$ is  the dual basis, then
\begin{equation*}
  \mathcal{L}^{*}_{\succ}(e_1)(e_1^*) = -2e_2^*,\;
  \mathcal{L}^{*}_{\succ}(e_2)(e_1^*) = -e_1^*-e_2^*,\;
  \mathcal{L}^{*}_{\succ}(e_2)(e_2^*) = -2e_2^*.
\end{equation*}
The non-zero multiplications of the left-Alia algebra $d=A\ltimes_{\mathcal{L}^{*}_{\succ},\mathcal{L}^{*}_{\succ}-\mathcal{R}^{*}_{\prec} }A^{*}$ are
\begin{align*}
  &[e_1,e_2]_d=e_1,\;[e_2,e_1]_d=-e_1,\; [e_1,e_1^*]_d=[e_2,e_2^*]_d=-2e_2^*, \;[e_2,e_1^*]_d=-e_1^*-e_2^*, \\
  &[e_1^*,e_1]_d=-4e_2^*, \;[e_1^*,e_2]_d=-2e_1^*-2e_2^*,\; [e_2^*,e_2]_d=-4e_2^*.
\end{align*}
Furthermore, by Theorem \ref{thm:6.8},
	\begin{equation*}
		r=\sum_{i=1}^{2}e^{*}_{i}\otimes e_{i}-e_{i}\otimes e^{*}_{i},
	\end{equation*}
	is an antisymmetric solution of the left-Alia Yang-Baxter equation in the left-Alia algebra $d$.
There exists a left-Alia bialgebra $(d,[\cdot,\cdot]_{d},\delta)$ with a non-zero co-multiplication $\delta$ on $d$, given by
	\begin{align*}
		&\delta(e_{1})=-4e_{2}^*\otimes e_{1}+2e_{2}^*\otimes e_{1}
+4e_{1}\otimes e_{2}^*+e_{2}^*\otimes e_{1}+2e_{1}\otimes e_{2}^*,\\
&\delta(e_{2})=-3e_{1}^*\otimes e_{1}-2e_{2}^*\otimes e_{1}
-4e_{1}\otimes e_{1}^*-2e_{1}\otimes e_{2}^*-2e_{2}^*\otimes e_{2}-4e_2\otimes e_{2}^*,\\
&\delta(e_{1}^*)=e_{1}^*\otimes e_{2}^*-e_{2}^*\otimes e_{1}^*.
	\end{align*}

\end{ex}



\noindent{\bf Acknowledgements.} The third author acknowledges support from the NSF China (12101328) and NSF China (12371039).

\end{document}